\documentclass[11pt,english]{smfart}

\usepackage[T1]{fontenc}
\usepackage[english,francais]{babel}
\usepackage[all,cmtip]{xy}
\usepackage{amsmath,amsfonts,amssymb,mathrsfs,latexsym,amscd,amsthm,url,xspace,smfthm}
\usepackage{color}
\usepackage{float}

\usepackage[utf8]{inputenc}

\usepackage{microtype}
\usepackage{framed}
\usepackage{tikz}
\usepackage{graphicx}
\usepackage{wasysym}

\usepackage{tikz-3dplot} 
\tdplotsetmaincoords{60}{125}
\tdplotsetrotatedcoords{80}{0}{0}

\usepackage{fancybox}

\makeatother

\newcommand{\BibTeX}{{\scshape Bib}\kern-.08em\TeX}
\newcommand{\T}{\S\kern .15em\relax }
\newcommand{\AMS}{$\mathcal{A}$\kern-.1667em\lower.5ex\hbox
        {$\mathcal{M}$}\kern-.125em$\mathcal{S}$}

\author{Zhizhong \textsc{Huang}}
\address{Univ. Grenoble Alpes, Institut Fourier, F-38000 Grenoble, France.}
\email{zhizhong.huang@univ-grenoble-alpes.fr}
\author{Pedro \textsc{Montero}}
\address{AMSS, Chinese Academy of Sciences, Beijing, People's Republic of China.}
\email{pmontero@amss.ac.cn}

\title[Additive Fano threefolds]{Fano threefolds as equivariant compactifications of the vector group}

\renewcommand\labelenumi{(\theenumi)}  

\newcommand{\Exc}{\operatorname{Exc}}

\newcommand{\Eff}{\operatorname{Eff}}

\newcommand{\NE}{\operatorname{NE}}

\newcommand{\Ga}{\mathbf{G}_a}

\newcommand{\II}{\operatorname{II}}
\newcommand{\III}{\operatorname{III}}
\newcommand{\IV}{\operatorname{IV}}
\newcommand{\V}{\operatorname{V}}

\theoremstyle{plain}
\newtheorem{thm}{Theorem}[section]
\newtheorem*{thm*}{Theorem}

\newtheorem*{mainthm*}{Main Theorem}

\newtheorem{lemma}[thm]{Lemma}
\newtheorem{cor}[thm]{Corollary}

\newtheorem{propo}[thm]{Proposition}

\theoremstyle{definition}
\newtheorem{defn}[thm]{Definition}

\newtheorem{exmp}[thm]{Example}

\newtheorem{prob}[thm]{Problem}
   
\newtheorem{ques}[thm]{Question}

\newtheorem{defn-thm}[thm]{Definition-Theorem}

\theoremstyle{remark}
\newtheorem*{claim*}{Claim}

\newtheorem{remark}[thm]{Remark}

\newtheorem*{not-and-def}{Notation and definitions}

\numberwithin{equation}{section}

\usepackage[square,sort,comma,numbers]{natbib}
\begin{document}

\def\smfbyname{}

\begin{abstract}
In this article, we determine all equivariant compactifications of the three-dimensional vector group $\mathbf{G}_a^3$ which are smooth Fano threefolds with Picard number greater or equal than two.
\end{abstract}

\maketitle

\tableofcontents

\newpage
\section{Introduction}

We work over the field of complex numbers $\mathbf{C}$.\\

The study of compactifications of the affine space $\mathbf{A}^n$ into smooth and complete algebraic varieties $X$ such that $B_2(X)=1$ was originally asked by Hirzebruch in \cite{Hir54}. In the projective case, which will be our case of interest, it is known after Kodaira \cite{Kod71} that, in this setting, $X$ has ample anticanonical divisor $-K_X$, i.e. $X$ is a {\it Fano manifold}. The problem to classify the pairs $(X,D)$ where $X\setminus D \cong \mathbf{A}^n$ has been solved in dimension $n\in\{1,2,3\}$: if $n=1$ we have $(X,D)\cong (\mathbf{P}^1,\{\mbox{point}\})$, and for $n=2$ Remmert and van de Ven proved in \cite{RV60} that $(X,D)\cong (\mathbf{P}^2,\mathbf{P}^1)$. The classification for $n=3$ has been achieved after the work of several authors \cite{BM78,Fur86,Fur90,Fur93a,Fur93b,FN89a,FN89b,Muk92,PS88,Pet89,Pet90,Pro91}: the possible pairs $(X,D)$ correspond to $\mathbf{P}^3$ with boundary a plane $\mathbf{P}^2$, the smooth quadric $\mathcal{Q}_3\subseteq \mathbf{P}^4$ with boundary a singular hyperplane section $\mathcal{Q}_0^2$, the Fano threefold $V_5$ with two possible boundaries and the Fano threefold $V_{22}$ with two possible boundaries. The case $n\geq 4$ remains open (see \cite{Pro94,PZ17} for partial results).

Through out this article we will be interested in some special smooth compactifications of $\mathbf{A}^n$ with $B_2\geq 2$. This problem has been considered by Morrow \cite{Mor72,Mor73} (cf. \cite{Kis02}) in the case of $n=2$, and by Kishimoto \cite{Kis05}, M\"{u}ller-Stach \cite{Mul90} and Nagaoka \cite{Nag17} in the case $n=3$ and $B_2=2$ for Fano threefolds. More precisely, we will treat the case of {\it additive Fano threefolds}:

By analogy with the case of toric varieties, where an algebraic torus $\mathbf{T}\cong \mathbf{G}_m^n$ operates effectively on a normal algebraic variety with a dense open orbit, we say that a normal projective variety $X$ is an {\it additive variety} if there is an effective action of the vector group $\mathbf{G}_a^n$ with a dense open orbit. In this case, $\dim(X)=n$ and the affine algebraic group $\mathbf{G}_a^n$ embeds equivariantly into $X$ as an open subset. In other words, additive varieties are projective equivariant compactifications of the algebraic group $\mathbf{G}_a^n$; for the latter we know that the underling scheme is isomorphic to the affine space $\mathbf{A}^n$. Therefore, by considering these varieties, we are situated in a rather special case of the previous situation. 

One of the main motivations for studying this class of algebraic varieties comes from arithmetic results due to Chambert-Loir and Tschinkel in \cite{CLT02,CLT12}, concerning the asymptotic distribution of rational points of bounded height on additive varieties defined over number fields. The precise statements concerning the arithmetic geometry of additive varieties and the relation with the Batyrev-Manin's principle are beyond the scope of this article, but we may refer the interested reader to the Bourbaki Seminar \cite{Pey02} for an introduction to this subject. From this point of view, it is interesting to find out explicit examples of additive varieties. The case of singular del Pezzo surfaces was treated by Derenthal and Loughran in \cite{DL10}. We would like to mention that a geometric analogue of the the Batyrev-Manin's principle is also established recently for smooth additive varieties over $\mathbf{C}$ by Bilu in her PhD thesis \cite{Bilu}.

The study of additive varieties began with the work of Hassett and Tschinkel in \cite{HT99}, where they establish some of the first properties of these varieties (see \S \ref{section:geometry}), they carefully study the case of $\mathbf{P}^n$ seen as an additive variety and they prove the so called {\it Hassett-Tschinkel correspondance}, which allows to identify all possible isomorphism classes of $\Ga^n$-structures in $\mathbf{P}^n$ with isomorphism classes of local commutative algebras of length $n+1$ (see \S \ref{section:HT} for more details). On one hand, it is well-known after Demazure \cite{Dem70} and Sumihiro \cite{Sumihiro} that toric structures on a normal variety are essentially unique: they are characterized by a fan (or polytope). On the other hand, it follows from Hassett-Tschinkel correspondance and a classical result of Suprunenko \cite{Sup56} on local commutative algebras that there are infinitely many non-isomorphic additive structures on $\mathbf{P}^n$ for $n\geq 6$. 

In order to obtain bounded families, we restrict ourselves to the case of Fano varieties\footnote{Smooth Fano varieties are bounded by \cite{KMM92}. See also \cite{Bir16a,Bir16b}.}. Our starting point is the following classification result proved by Hassett and Tschinkel in \cite[Theorem 6.1]{HT99} (cf. \S \ref{section:HT} and \S \ref{section:quadrics}).

\begin{thm}[Hassett \& Tschinkel]\label{theo:HT} Let $X$ be a smooth projective threefold which is an equivariant compactification of $\mathbf{G}_a^3$ and such that ${B_2(X)=1}$. Then $X \cong \mathbf{P}^3$ or $X \cong \mathcal{Q}_3 \subseteq \mathbf{P}^4$.
\end{thm}

In regard of Theorem \ref{theo:HT} and the previous discussion, the main purpose of this article is to treat the case of smooth Fano threefolds with $B_2\geq 2$. All these threefolds were classified into 88 families by Mori and Mukai in \cite{MM81,MM03}. Our main result, which can be seen as a more precise version of a birational classification of Kebekus \cite{Keb98} in this particular case, reads as follows.

\begin{mainthm*}\label{theo:main}
Let $X$ be a smooth Fano threefold which is an equivariant compactification of the additive group $\Ga^3$. Then $X$ is isomorphic to one of the following.
\begin{enumerate}
 \item $\mathbf{P}^3$.
 \item The non-degenerate quadric $\mathcal{Q}_3 \subseteq \mathbf{P}^4$.
 \item $\II_{28}$, the blow-up of $\mathbf{P}^3$ along a plane cubic.
 \item $\II_{30}$, the blow-up of $\mathbf{P}^3$ along a conic or, equivalently, the blow-up of $\mathcal{Q}_3 \subseteq \mathbf{P}^4$ at a point.
 \item $\II_{31}$, the blow-up of $\mathcal{Q}_3 \subseteq \mathbf{P}^4$ along a line.
 \item $\II_{33}$, the blow-up of $\mathbf{P}^3$ along a line or, equivalently, the projective bundle $\mathbf{P}(\mathcal{O}_{\mathbf{P}^1}^{\oplus 2}\oplus \mathcal{O}_{\mathbf{P}^1}(1))$.
 \item $\II_{34} \cong \mathbf{P}^1 \times \mathbf{P}^2$.
 \item $\II_{35}$, the blow-up of $\mathbf{P}^3$ at a point $p$ or, equivalently, the projective bundle $\mathbf{P}(\mathcal{O}_{\mathbf{P}^2}\oplus \mathcal{O}_{\mathbf{P}^2}(1))$.
 \item $\II_{36}\cong \mathbf{P}(\mathcal{O}_{\mathbf{P}^2}\oplus \mathcal{O}_{\mathbf{P}^2}(2))$.
 \item $\III_{23}$, the blow-up of $\mathbf{P}(\mathcal{O}_{\mathbf{P}^2}\oplus \mathcal{O}_{\mathbf{P}^2}(1))$ with center a conic passing through the center of the blow-up $\mathbf{P}(\mathcal{O}_{\mathbf{P}^2}\oplus \mathcal{O}_{\mathbf{P}^2}(1))\to \mathbf{P}^3$.
 \item $\III_{26}$, the blow-up of $\mathbf{P}^3$ with center a disjoint union of a point and a line.
 \item $\III_{27}\cong\mathbf{P}^1 \times \mathbf{P}^1 \times \mathbf{P}^1.$
 \item $\III_{28}\cong \mathbf{P}^1\times \mathbf{F}_1$. 
 \item $\III_{29}$, the blow-up of $\mathbf{P}(\mathcal{O}_{\mathbf{P}^2}\oplus \mathcal{O}_{\mathbf{P}^2}(1))$ with center a line on the exceptional divisor of the blow-up $\mathbf{P}(\mathcal{O}_{\mathbf{P}^2}\oplus \mathcal{O}_{\mathbf{P}^2}(1))\to \mathbf{P}^3$.
 \item $\III_{30}$, the blow-up of $\mathbf{P}(\mathcal{O}_{\mathbf{P}^2}\oplus \mathcal{O}_{\mathbf{P}^2}(1))$ with center the strict transform of a line passing through the center of the blow-up $\mathbf{P}(\mathcal{O}_{\mathbf{P}^2}\oplus \mathcal{O}_{\mathbf{P}^2}(1))\to \mathbf{P}^3$.
 \item $\III_{31}\cong \mathbf{P}(\mathcal{O}_{\mathbf{P}^1\times \mathbf{P}^1}\oplus \mathcal{O}_{\mathbf{P}^1\times \mathbf{P}^1}(1,1))$.
 \item $\IV_{10} \cong \mathbf{P}^1 \times \mathbf{S}_7$.
 \item $\IV_{11}$, the blow-up of $\mathbf{P}^1 \times \mathbf{F}_1$ with center $\{t\}\times \{E\}$, where $t\in \mathbf{P}^1$ and $E\subseteq \mathbf{F}_1$ is an exceptional curve of the first kind.
 \item $\IV_{12}$, the blow-up of $\mathbf{P}(\mathcal{O}_{\mathbf{P}^1}^{\oplus 2}\oplus \mathcal{O}_{\mathbf{P}^1}(1))$ with center two exceptional lines of the blow-up $\mathbf{P}(\mathcal{O}_{\mathbf{P}^1}^{\oplus 2}\oplus \mathcal{O}_{\mathbf{P}^1}(1))\to \mathbf{P}^3$.
\end{enumerate}
Conversely, every Fano threefold in the above listed families\footnote{For $\III_{23}$ this is true only for a particular member of the family (see Lemma \ref{lemma:III23} and its proof).} admits a $\Ga^3$-structure. 
\end{mainthm*}

Finally, let us remark that we only treat the problem of existence of additive structures on smooth Fano threefolds. The methods and the dependence between the sections of this article is discussed below. 

The unicity problem is more subtle (cf. Remark \ref{remark:unicity quadric} and Remark \ref{remark: unicity flag}) and even for toric varieties, for which is known after the work of Arzhantsev and Romaskevich \cite{AR17} that the existence of an additive structure is equivalent to existence of a (unique) structure normalized by the acting torus (see Theorem \ref{theo: toric existence}), in general is not clear for the authors whether there are finitely many non-normalized actions or not (cf. Lemma \ref{lemma:product actions}, \cite[Proposition 5.5]{HT99} and \cite[Problem 2]{AR17}).

The varieties listed in Main Theorem can be represented by the following family tree of additive Fano threefolds, in which the toric ones are double-boxed and two of them are connected by a solid line (resp. dashed line) if one can be obtained as the blow-up of the other one along a curve (resp. a point). A posteriori, all primitive additive Fano threefolds with $B_2(X)\geq 2$ are toric, and every primitive toric Fano threefold is additive (see Corollary \ref{primitive toric}).
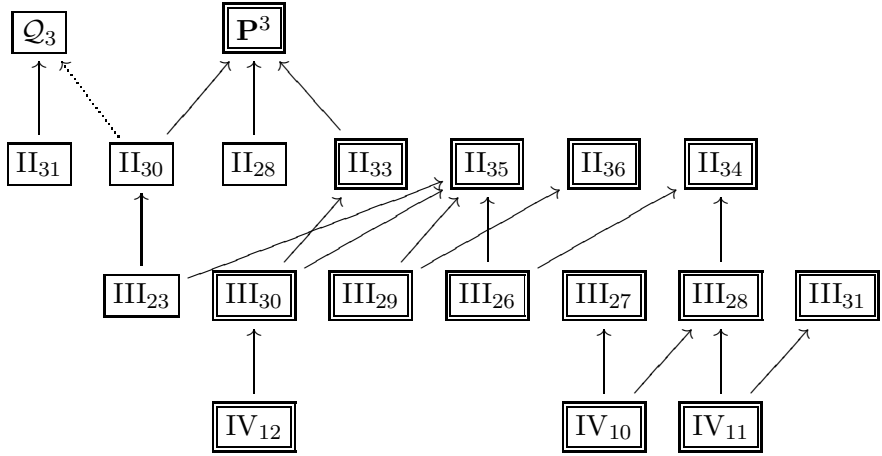
$$
\xymatrixcolsep{2mm}
\xymatrix{
\boxed{\mathcal{Q}_3} & & \doublebox{$\mathbf{P}^3$} \\
\boxed{\II_{31}} \ar[u] & \boxed{\II_{30}} \ar@{.>}[ul] \ar[ur] & \boxed{\II_{28}} \ar[u] & \doublebox{$\II_{33}$} \ar[ul] & \doublebox{$\II_{35}$} & \doublebox{$\II_{36}$} & \doublebox{$\II_{34}$} \\
& \boxed{\III_{23}} \ar[u] \ar[urrr] & \doublebox{$\III_{30}$} \ar[ur] \ar[urr] & \doublebox{$\III_{29}$}  \ar[ur] \ar[urr] & \doublebox{$\III_{26}$} \ar[u] \ar[urr] & \doublebox{$\III_{27}$} & \doublebox{$\III_{28}$} \ar[u] & \doublebox{$\III_{31}$}\\
& & \doublebox{$\IV_{12}$} \ar[u] & & & \doublebox{$\IV_{10}$} \ar[u] \ar[ur] & \doublebox{$\IV_{11}$} \ar[u] \ar[ur] &
}$$

\subsection*{Methods and outline of the article}

In section \S \ref{section:geometry} we review some of the properties of {\it additive varieties} and {\it additive morphisms}. The central part of the article is \S \ref{section:Fano 3folds} where we prove our Main Theorem, and the dependence with the previous sections goes as follows: In order to study some explicit blow-ups we need to work in coordinates. In particular, we recall in \S \ref{section:HT} the {\it Hassett-Tschinkel correspondance} which allows to explicitly describe all additive structures on the projective space $\mathbf{P}^3$, while the the case of smooth and singular quadric hypersurfaces is treated in \S \ref{section:quadrics}. The toric case is treated separately in \S \ref{section:toric} and Appendix \ref{appendix}. Arzhantsev's criterion for flag varieties is stated in \S \ref{section:flags} and used later on to prove that the projective bundle $\mathbf{P}(T_{\mathbf{P}^2})$ is not additive. In \S \ref{section:kishimoto} we review Kishimoto's classification of smooth Fano threefolds of Picard number two which are compactification of $\mathbf{A}^3$ with associated log canonical divisor being not nef. We use Kishimoto's classification in a very essential way during the article. Finally, we discuss in \S \ref{section:open problems} some open problems and questions.

\section*{Notation}\label{notation}
For a locally free sheaf $E$ on a variety $X$ we define the projectivization $\mathbf{P}(E)$ to be $\mathbf{Proj}_{\mathcal{O}_X}\oplus_{m\geq 0} \mbox{Sym}^m(E)$, following Grothendieck's convention. We denote by $\mathbf{F}_a$ the Hirzebruch surface $\mathbf{P}(\mathcal{O}_{\mathbf{P}^1}\oplus \mathcal{O}_{\mathbf{P}^1}(a))$, with $a\geq 0$. We denote by $\mathbf{S}_d$ the smooth del Pezzo surface of degree $(K_{\mathbf{S}_d})^2=d\in \{1,\ldots,8\}$, where $\mathbf{S}_8 \cong \mathbf{F}_1$. Along the paper $\mathcal{Q}_n\subseteq \mathbf{P}^{n+1}$ stands for the non-degenerate quadric hypersurface of dimension $n$, while $\mathcal{Q}_0^n \subseteq \mathbf{P}^{n+1}$ denotes the corank one quadric hypersurface of dimension $n$, i.e. the projective cone over $\mathcal{Q}_{n-1}\subseteq \mathbf{P}^n$. We will freely use the notation and results of the Minimal Model Program (MMP for short) in \cite{KM98}. \\

We will denote by Roman numerals the Picard number of a family of Fano threefolds and by a sub-index the corresponding number in the tables in \cite{MM81,MM03}. For example, $\operatorname{II}_{35}$ stands for the (deformation class of the) Fano threefold \mbox{No. 35} in the list of Mori and Mukai for Fano threefolds of Picard number $2$ (i.e., $\operatorname{II}_{35} \cong \mathbf{P}(\mathcal{O}_{\mathbf{P}^2}\oplus \mathcal{O}_{\mathbf{P}^2}(1))$ in \cite{MM81}) and $\IV_{13}$ stands for the family of Fano threefolds \mbox{No. 13} in the list of Mori and Mukai for Fano threefolds of Picard number $4$ (i.e. the Fano threefold family missing in \cite{MM81} and appearing in \cite{MM03}, whose elements are obtained as the blow-up of $\mathbf{P}^1\times \mathbf{P}^1 \times \mathbf{P}^1$ with center a curve of tridegree $(1,1,3)$). We will sometimes, by abuse of notation, regard the families as single varieties if it is clear from the context (e.g. by {\it the variety} $\II_{27}$ we mean {\it any member of the family} $\II_{27}$). \\

\textbf{Convention.} Following \cite{MM81}, we say that a smooth Fano threefold $X$ with $B_2(X)\geq 2$ is {\it primitive} if $X$ cannot be realized as the blow-up of a smooth Fano threefold along a curve.

\section{Geometry of additive varieties and equivariant morphisms}\label{section:geometry}

In general, the category of additive varieties together with their additive structures and morphisms can be seen (via the forgetful functor) as a subcategory of the category of compactifications of the affine space with their corresponding boundary divisor. Since the additive structures are in general not unique, we need to keep track of the additive structure on these varieties. This section is devoted to recall some special features concerning additive varieties and additive morphisms. They are distributed in various references so for the sake of the reader sometimes we give proofs.
\begin{propo}[{\cite[Theorem 2.5]{HT99}}]\label{prop:boundary divisors}
Let $X$ be a normal projective equivariant compactification of $\Ga^n$ of Picard number $\rho$. Then the boundary is given by $X\setminus \Ga^n=\cup_{i=1}^{\rho} D_i$ where $D_i$ are all divisors. The group $\operatorname{Cl}(X)$ is freely generated by $[D_i],1\leqslant i\leqslant \rho$ and $\Eff(X)=\sum_{i=1}^\rho \mathbf{R}_{\geq 0}[D_i]$.
\end{propo}
\begin{proof}
Write $U=\mathbf{A}^n$ and $X\setminus U=\cup_{i=1}^m D_i$ where $D_i,1\leqslant i\leqslant m$ are irreducible components. Suppose there exists say $D_1$ of codimension at least two. Take $V$ an affine open subset of $X\setminus\cup_{i=2}^m D_i$ such that $V\cap D_1$ is non-empty and $V\cap U$ is also affine. Then by Hartog's extension theorem or Krull's Hauptidealsatz (cf. \cite[Proposition I.6.3A]{Hartshorne}), the restriction morphism
$$\mathcal{O}_X(V)\to\mathcal{O}_X(V\cap U)$$
is an isomorphism. But $V\cap U$ is clearly a proper open subscheme of $V$, whose ring spectrum is obtained by taking localization at the prime ideal defining $V\cap D_1$, which is a contradiction. So all $D_i$'s are divisors. 
Since $\mathcal{O}_X(U)$ is factorial, any (effective) divisor is linearly equivalent, by adding a principal divisor $\operatorname{div}(f)$ with $f\in\mathbf{C}[X_1,\cdots,X_n]$, to another one disjoint from $U$. Since the only regular invertible functions on $U$ are constant, there is no non-trivial relation between $[D_i],1\leqslant i\leqslant m$ and $m=\rho$.
\end{proof}

By looking at vanishing order of vector fields along the boundary components and using the fact that the additive group $\Ga$ has no characters, Hassett and Tschinkel proved in \cite[Theorem 2.7]{HT99} the following result concerning the canonical class of smooth additive varieties.

\begin{thm}[Hassett \& Tschinkel]\label{theo:canonical divisor}
Let $X$ be a smooth projective equivariant compactification of $\Ga^n$ of Picard number $\rho$. Then $$-K_X=\sum_{i=1}^{\rho} a_i[D_i],\quad a_i\geqslant 2,$$
where $D_1,\ldots,D_\rho$ denote the divisors on the boundary $X\setminus \Ga^n$. In particular, the log canonical divisor $K_X+\sum_i D_i $ is not nef and smooth projective compactifications of $\Ga^n$ of Picard number 1 are Fano manifolds of Fano index  $i_X \geq 2$.
\end{thm}

\begin{proof}
The first part of the statement is the content of \cite[Theorem 2.7]{HT99}. It follows from \cite[Lemma 4.6]{LM09} that in this case the log canonical divisor $K_X+\sum_i D_i$ is not pseudoeffective and hence not nef. The fact that smooth projective compactification of the affine space of Picard number 1 are Fano was observed by Kodaira in \cite{Kod71}.
\end{proof}


Next, we recall the following algebraic version of Blanchard's lemma \cite[Proposition I.1]{Bla56} proved by Brion in \cite[Theorem 7.2.1]{Bri17}, which will allow us to run a $\Ga^n-$MMP. See also \cite[Corollary 2.4]{HT99} and \cite[Corollary 7.2.2]{Bri17}.

\begin{thm}[Blanchard's lemma]\label{theo:Blanchard} Let $G$ be a connected linear algebraic group, $X,Y$ be normal varieties and $f:X\to Y$ be a proper morphism such that $f_*\mathcal{O}_X=\mathcal{O}_Y$. Then any $G$-structure on $X$ uniquely determines a $G$-structure on $Y$ such that $f$ is $G$-equivariant.
\end{thm}

As a consequence, we obtain the following correspondence. See also \cite[Lemmata 2.5, 2.6, 2.7]{DL15}.

\begin{cor}\label{co:equivariant blow-up}
	 Let $G$ be a connected linear algebraic group, $X$ be a normal variety and $Z$ a subvariety of $X$. Then there is an one-to-one correspondence between $G$-structures on $\operatorname{Bl}_Z X$ and those on $X$ such that $G$ acts on $Z$.
\end{cor}
\begin{proof}
	We have by Zariski's Main Theorem \cite[Corollary III.11.4]{Hartshorne} that $f_*\mathcal{O}_{\operatorname{Bl}_Z X}=\mathcal{O}_X$. By Blanchard's lemma, any $G$-structure descends in a unique way to $X$. Since any such structure has to fix the exceptional divisor, so the descended one fixes $Z$. Conversely, for any $G$-structure on $X$ fixing $Z$, we have that the inverse image of the ideal sheaf $\mathcal{I}_Z$ by $$\operatorname{Bl}_Z X\times G\to X\times G\to X$$ is the invertible sheaf $\mathcal{O}(1)$. So there exists a unique lifting of this $G$-structure to $\operatorname{Bl}_Z X$ by the universal property of blow-up.
\end{proof}
For $X$ a variety, let $\operatorname{Aut}^0(X)$ denote the connected component of automorphism group scheme of $X$ containing the identity.
\begin{cor}[{\cite[Corollary 7.2.3]{Bri17}}]\label{co:productauto}
	Let $X,Y$ be proper varieties and let $p_1,p_2$ be the projections from $X\times Y$ to $X$ and $Y$, respectively. Then the natural morphism
	$$p_{1*}\times p_{2*}:\operatorname{Aut}^0(X\times Y)\to \operatorname{Aut}^0(X)\times \operatorname{Aut}^0(Y)$$
	is an isomorphism.
\end{cor}
\begin{proof}
	We have $$p_{1*}(\mathcal{O}_{X\times Y})=\mathcal{O}_X,\quad  p_{2*}(\mathcal{O}_{X\times Y})=\mathcal{O}_Y$$
	since both $X$ and $Y$ are proper varieties. We can apply Blanchard's lemma and \cite[Corollary 7.2.2(1)]{Bri17} to the projections $p_1$ and $p_2$ in order to get the desired isomorphism.
\end{proof}

\section{Hassett-Tschinkel correspondance}\label{section:HT}

Let us begin by recalling some results in \cite[\S 2.4]{HT99}, where Hassett and Tschinkel establish a dictionary between cyclic representations of $\mathbf{G}_a^n$ and finite-dimensional local algebras. We follow the presentation in \cite[\S 2]{Sha09} and \cite[\S 1]{AS11}.

Let us denote by $\mathfrak{g}=\operatorname{Lie}(\Ga^n)$ the Lie algebra of $\Ga^n$ and by $\mathfrak{U}(\mathfrak{g})$ its universal enveloping algebra. Since the algebraic group $\Ga^n$ is commutative, the algebra $\mathfrak{U}(\mathfrak{g})$ coincides with the symmetric algebra $\mathfrak{S}(\mathfrak{g})$. In other words, if we choose a basis of $\mathfrak{g}$ given by $S_1=\frac{\partial}{\partial x_1},\ldots, S_n=\frac{\partial}{\partial x_n}$ then there is an isomorphism $\mathfrak{U}(\mathfrak{g})\cong \mathbf{C}[S_1,\ldots,S_n]$ such that $\mathfrak{g}$ is identified with the subspace $\langle S_1,\ldots, S_n \rangle$.

Let $\rho:\Ga^n \to \operatorname{GL}_\ell(\mathbf{C})$ be a faithful algebraic representation and let ${\tau: \mathfrak{U}(\mathfrak{g})\to \operatorname{Mat}_{\ell}(\mathbf{C})}$ be the induced representation of $\mathfrak{U}(\mathfrak{g})$. The algebra $R=\tau(\mathfrak{U}(\mathfrak{g})) \cong \mathfrak{U}(\mathfrak{g})/\ker(\tau)$ is local of finite dimension, since $\tau(S_1),\ldots,\tau(S_n)$ are commuting nilpotent elements which generate this algebra. If we consider $X_1=\tau(S_1),\ldots,X_n=\tau(S_n)$ in $R$ then the maximal ideal of $R$ is given by $\mathfrak{m}_R=(X_1,\ldots,X_n)$ and the subspace $U=\tau(\mathfrak{g})=\langle X_1,\ldots,X_n\rangle $ generates the algebra $R$.

The Hassett-Tschinkel correspondence goes as follows: Given a faithful algebraic representation $\rho:\Ga^n \to \operatorname{GL}_\ell(\mathbf{C})$ and a vector $v\in \mathbf{C}^\ell$ such that $\langle \rho(\Ga^n)v\rangle = \mathbf{C}^\ell$, i.e., a {\it cyclic vector}, we have that $\ker(\tau)$ is given by the ideal $I=\{x\in \mathfrak{U}(\mathfrak{g})\;|\;\tau(x)v=0\}$ and hence $R\cong \mathfrak{U}(\mathfrak{g})/I\cong \tau(\mathfrak{U}(\mathfrak{g}))v = \mathbf{C}^\ell$. The action of an element $\tau(x)$ on $\mathbf{C}^\ell$ corresponds via these isomorphisms to the multiplication by $\tau(x)$ on $R$, and the cyclic vector $v\in \mathbf{C}^\ell$ corresponds to the unity of the ring $R$. Moreover, since $\Ga^n=\exp(\mathfrak{g})$ we have that the action of $\Ga^n$ on $\mathbf{C}^\ell$ corresponds to the multiplication by elements of $\exp(U)$ on $R$.

Conversely, given a commutative local algebra $(R,\mathfrak{m}_R)$ supported at the origin of dimension $\ell$ and $U\subseteq \mathfrak{m}_R$ a subspace that generates $R$ as an algebra, we can define a additive faithful representation by considering a basis $X_1,\ldots,X_n$ of $U$ and therefore getting an isomorphism $R\cong \mathbf{C}[S_1,\ldots,S_n]/I$ which is induced by the surjective homomorphism $\mathbf{C}[S_1,\ldots,S_n]\to R,\;S_i\to X_i$. In fact, we get a faithful representation by considering $\rho:\Ga^n=\exp(U)\to R\cong \mathbf{C}^\ell$ such that $\rho((a_1,\ldots,a_n))$ acts via multiplication by $\exp(a_1X_1+\ldots+a_nX_n)$ on $R$. In this way, the unit of $R$ corresponds to a cyclic vector, since $U$ generates the algebra $R$. 

The above correspondence is summarized in \cite[Theorem 2.14]{HT99} and reads as follows.

\begin{thm}[Hassett \& Tschinkel] There is a bijection between the following:
\begin{enumerate}
 \item Isomorphism classes of pairs $(\rho,v)$, where $\rho: \Ga^n\to \operatorname{GL}_\ell(\mathbf{C})$ is a faithful representation and $v$ is a cyclic vector;
 \item Isomorphism classes of pairs $(R,U)$, where $(R,\mathfrak{m}_R)$ is a local algebra supported at the origin of dimension $\ell$ and $U$ is an $n$-dimensional linear subspace of $\mathfrak{m}_R$ that generates $R$ as an algebra.
 \end{enumerate}
\end{thm}

Moreover, a faithful algebraic representation $\rho:\Ga^n\to \operatorname{GL}_\ell(\mathbf{C})$ corresponds to a unique effective $\Ga^n$-action on the projective space $\mathbf{P}^{\ell-1}$, and vice-versa (cf. \cite[Lemma 2.3]{DL15}). In particular, an effective action of the additive group $\Ga^n$ on $\mathbf{P}^{\ell-1}$ is generically transitive if and only if $n=\ell-1$. As a consequence, Hassett and Tschinkel obtain the following characterization of generically transitive actions of the additive group on the projective space.

\begin{propo}[{\cite[Proposition 2.15]{HT99}}] There is a bijection between the following:
\begin{enumerate}
 \item Isomorphism classes of generically transitive $\Ga^n$-actions on $\mathbf{P}^n$;
 \item Isomorphism classes of local commutative algebras of dimension $n+1$.
\end{enumerate}
\end{propo}

The classification of local commutative algebras of dimension 3 leads to the following result.

\begin{propo}[{\cite[Proposition 3.2]{HT99}}]\label{prop:P2} There are two distinct $\mathbf{G}_a^2$-structures on $\mathbf{P}^2$. They are given by the following representations of $\mathbf{G}_a^2$:
$$\tau(a_1,a_2)=
\begin{pmatrix}
1 & 0 & a_2 \\
0 & 1 & a_1 \\
0 & 0 & 1
\end{pmatrix} \text{ and  } \rho(a_1,a_2)=
\begin{pmatrix}
1 & a_1 & a_2+\frac{1}{2}a_1^2 \\
0 & 1 & a_1 \\
0 & 0 & 1
\end{pmatrix}.$$
In particular, the fixed locus are given by a line for $\tau$ and a point for $\rho$.
\end{propo}

Let us mention that there is an inaccuracy\footnote{See also \cite[Lemma 1]{Sha09} for a correction in the classification of local commutative algebras of dimension 5. We also refer the interested reader to \cite{Poo08}.} in the list \cite[Proposition 3.3]{HT99} of local commutative algebras of dimension 4. The right list is given by the following ideals.

\begin{propo} There are four distinct $\mathbf{G}_a^3$-structures on $\mathbf{P}^3$. They correspond to the quotients of $\mathbf{C}[S_1,S_2,S_3]$ by the following ideals:
$$
\begin{array}{l}
I_1 = \langle S_1^2-S_2,S_1S_2-S_3,S_1S_3 \rangle \\
I_2 = \langle S_1^2 - S_2, S_1S_2,S_1S_3, S_3^2 \rangle \\
I_3 = \langle S_1^2,S_1S_2-S_3,S_2^2 \rangle \\
I_4 = \langle S_1^2, S_1S_2, S_2^2, S_2S_3, S_3^2, S_1S_3 \rangle.
\end{array}
$$
\end{propo}

\begin{proof}
We keep the same notation as in \cite{Poo08}. All possible local commutative algebras $(R,\mathfrak{m}_R)$ of dimension $n=4$ are listed in \cite[Table 1]{Poo08} and they are classified depending on the possible values $d_i=\dim_\mathbf{C}(\mathfrak{m}_R^i\slash \mathfrak{m}_R^{i+1})$ of their Hilbert-Samuel function. It follows from \cite[Lemma 1.3]{Poo08} that if the tangent space of $R$ has dimension $d_1=1$ then
$$R\cong \mathbf{C}[S_1,S_2,S_3]\slash I_1 \cong \mathbf{C}[x]\slash \langle x^4 \rangle. $$
Similarly, if $d_1=3$ then
$$R\cong \mathbf{C}[S_1,S_2,S_3]\slash I_4 \cong \mathbf{C}[x,y,z]\slash \langle x,y,z\rangle^2. $$
The remaining case is $(d_1,d_2)=(2,1)$, for which we have that either
$$R\cong \mathbf{C}[S_1,S_2,S_3]\slash I_2 \cong \mathbf{C}[x,y]\slash \langle x^2,xy,y^3 \rangle $$
or
$$R\cong \mathbf{C}[S_1,S_2,S_3]\slash I_3 \cong \mathbf{C}[x,y]\slash \langle x^2,y^2 \rangle. $$
The result follows from \cite[Table 1]{Poo08}
\end{proof}

The correspondence above leads therefore to the following result.

\begin{cor}\label{coro:P3} There are four distinct $\mathbf{G}_a^3$-structures on $\mathbf{P}^3$. They are given by the following representations of $\mathbf{G}_a^3$:
$$
\begin{array}{l}
\rho_1(a_1,a_2,a_3)=
\begin{pmatrix}
1 & a_1 & a_2+\frac{1}{2}a_1^2 & a_3 + a_1a_2 \\
0 & 1 & a_1 & a_2+\frac{1}{2}a_1^2 \\
0 & 0 & 1 & a_1 \\
0 & 0 & 0 & 1 
\end{pmatrix}\\
\rho_2(a_1,a_2,a_3)=
\begin{pmatrix}
1 & 0 & 0 & a_3 \\
0 & 1 & a_1 & a_2+\frac{1}{2}a_1^2 \\
0 & 0 & 1 & a_1 \\
0 & 0 & 0 & 1 
\end{pmatrix}\\
\rho_3(a_1,a_2,a_3)=
\begin{pmatrix}
1 & a_1 & a_2 & a_3 + a_1a_2 \\
0 & 1 & 0 & a_2 \\
0 & 0 & 1 & a_1 \\
0 & 0 & 0 & 1 
\end{pmatrix}\\ 
\rho_4(a_1,a_2,a_3)=
\begin{pmatrix}
1 & 0 & 0 & a_3 \\
0 & 1 & 0 & a_2 \\
0 & 0 & 1 & a_1 \\
0 & 0 & 0 & 1 
\end{pmatrix}.
\end{array}$$
In particular, the fixed locus are given by a point for $\rho_1$ and $\rho_3$, a line for $\rho_2$ and a plane for $\rho_4$.
\end{cor}
\begin{propo}\label{prop:Gastructureonproj}
If the vector group $\Ga^n$ acts transitively (generically) on the $\mathbf{P}^l$ with $n \geqslant l+1$, then there is a subgroup $G$ of $\Ga^n$ isomorphic to $\Ga^{n-l}$ such that the action of $\Ga^n$ factorizes via $\Ga^n/G\simeq \Ga^l$ on $\mathbf{P}^l$. In fact we have $G=\operatorname{Stab}_x \Ga^n$ where $x$ is a general point in $\mathbf{P}^l$. 
\end{propo}
\begin{proof}
In view of the Hassett-Tschinkel correspondence, the action of $\Ga^n$ on $\mathbf{P}^l$ induces a representation $\Ga^n\to \operatorname{GL}_{l+1}(\mathbf{C})$ which corresponds to an Artin local algebra $R=\mathbf{C}[X_1,\cdots,X_n]/I$ of length $l+1$, where $I$ is an ideal such that $\sqrt{I}$ is the maximal ideal supported at the origin. Write $\overline{X_i}$ the class of $X_i$ in $R$. The action of an element $(a_1,\cdots,a_n)\in \Ga^n$ on the homogeneous coordinates corresponds to the linear transformation sending a basis $\{S_1,\cdots,S_{l+1}\}$ of $R$ into the product of each element of this base with the element $\operatorname{exp}(a_1\overline{X_1}+\cdots+a_n\overline{X_n}) \in R$.

Up to permutation, let $\{1,\overline{X_1},\cdots,\overline{X_m}\}$ be a maximal linearly independent collection in $R$. We claim that $m=l$. Indeed, write
\begin{equation}\label{eq:relation1}
\overline{X_k}=\sum_{i=1}^m b_{k,i} \overline{X_i},\quad m+1\leqslant k\leqslant n, b_{k,i}\in\mathbf{C}.
\end{equation}
(the coefficient of $1$ must be $0$ otherwise $\overline{X_k}$ would be invertible, which is impossible since some power of $X_k$ is contained in $I$.)
Then we have
\begin{align*}
\operatorname{exp}\left(\sum_{i=1}^n a_i \overline{X_i}\right)&=\operatorname{exp}\left(\sum_{i=1}^m a_i \overline{X_i}+\sum_{i=m+1}^n a_i\left(\sum_{j=1}^m b_{i,j}\overline{X_j}\right)\right)\\
&=\operatorname{exp}\left(\sum_{i=1}^m \left(a_i+\sum_{k=m+1}^n a_k b_{k,i}\right) \overline{X_i}\right)
\end{align*}
We see that the action of $\Ga^n$ factorizes into an action of $\Ga^{m}$ on $\mathbf{P}^l$ with kernel defined by the equations
$$a_i+\sum_{k=m+1}^n a_k b_{k,i}=0,\quad \forall i\in\{1,\cdots,m\}.$$
This is a vector group of dimension $n-m$ serving as the stabilizer.
So this action cannot be generically transitive unless $m\geqslant l$. This proves the claim and finishes the proof.
\end{proof}

As a consequence we get the following result concerning additive actions on $\mathbf{P}^1 \times \mathbf{P}^2$ and $\mathbf{P}^1 \times \mathbf{P}^1 \times \mathbf{P}^1$.

\begin{lemma}\label{lemma:product actions}
There exist, up to isomorphism, a unique $\Ga^3$-structure on $\mathbf{P}^1 \times \mathbf{P}^1 \times \mathbf{P}^1$ and two different $\Ga^3$-structures on $\mathbf{P}^1 \times \mathbf{P}^2$.
\end{lemma}
\begin{proof}
The result follows from Corollary \ref{co:productauto} and Proposition \ref{prop:Gastructureonproj}, together with the fact that there is a unique additive structure on $\mathbf{P}^1$ (see \cite[Proposition 3.1]{HT99} for instance) and that they are two additive structures on $\mathbf{P}^2$, by Proposition \ref{prop:P2}.
\end{proof}

\section{Additive actions on quadric hypersurfaces}\label{section:quadrics}
\subsection{Non-degenerate quadrics}\label{section:quadrics smooth}

It is known since the work of Hassett and Tschinkel \cite[\S 3]{HT99} that the number of isomorphism classes of $\Ga^n$-structures on $\mathbf{P}^n$ is finite if and only if $n\leq 5$. The problem concerning the classification of all $\Ga^n$-structures on the non-degenerate quadric $\mathcal{Q}_n \subseteq \mathbf{P}^{n+1}$ was considered by Sharoyko in \cite{Sha09}. A full answer to the classification problem is given by the following result, which summarizes \cite[Theorem 4]{Sha09} and the discussion in \cite[pages 1726-1727]{Sha09}.

\begin{thm}[Sharoyko]\label{theo:smooth quadric} 
For each integer $n\geq 1$ there exists a unique $\Ga^n$-structure on the non-degenerate quadric $\mathcal{Q}_n \subseteq \mathbf{P}^{n+1}$ up to isomorphism. Moreover, there are homogeneous coordinates $[x_0:\ldots:x_{n+1}]$ of $\mathbf{P}^{n+1}$ such that
$$\mathcal{Q}_n=\left\{[x_0,\ldots,x_{n+1}]\in \mathbf{P}^{n+1}\;\left|\;\sum_{i=1}^n x_i^2 = 2x_0x_{n+1} \right.\right\} $$
and the $\Ga^n$-structure is given by the representation
$$\rho(a_1,\ldots,a_n)= \begin{pmatrix}
1 & 0 & \cdots & 0 & 0 \\
a_1 & 1 & \ddots & \ddots & 0 \\
\vdots & \vdots & \ddots & \ddots & \vdots \\
a_n & 0 & \cdots & 1 & 0 \\
\frac{1}{2}\sum_{i=1}^n a_i^2 & a_1 & \cdots & a_n & 1
\end{pmatrix}$$
In particular, the fixed locus of $\rho$ is given by the single point $[0:\ldots:0:1]\in \mathbf{P}^{n+1}$.
\end{thm}

\begin{remark}\label{remark:unicity quadric}
 The above uniqueness result has been generalized by Fu and Hwang in \cite[Theorem 1.2, Corollary 1.3]{FH14} to a larger class of Fano manifolds of Picard number one. See also \cite{Dev15} and \cite{FH17}.
\end{remark}

In the case of threefolds the result of Sharoyko leads to the following (cf. Proof of \cite[Theorem 6.1]{HT99}).

\begin{cor}\label{cor:smooth quartic 3fold}
The boundary divisor for the unique $\Ga^3$-structure on the smooth quadric $\mathcal{Q}_3 \subseteq \mathbf{P}^4$ is given by a singular quadric hyperplane section $\mathcal{Q}_0^2 \subseteq \mathcal{Q}_3$ and the only invariant curves contained in $\mathcal{Q}_0^2$ are the lines of the distinguished ruling passing through the (isolated) singular point of $\mathcal{Q}_0^2$.
\end{cor}

\subsection{Degenerate quadrics}\label{section:quadrics singular} Arzhantsev and Popovskiy proved in \cite{AP14} that additive actions on projective hypersurfaces correspond to invariant multilinear symmetric forms on local algebras. As an application, they classify additive actions on quadrics of corank one. We will only address the case of three-dimensional quadrics of corank one, and we refer the interested reader to \cite[\S 6]{AP14} for the general case.

\begin{thm}[Arzhantsev \& Popovskiy]\label{theo:degenerate quadric} Let $\mathcal{Q}_0^3$ be the quadric cone in $\mathbf{P}^4$ over a smooth quadric surface $\mathcal{Q}_2\subseteq \mathbf{P}^3$, given by
$$\mathcal{Q}_0^3= \{ [x_0:x_1:x_2:x_3:x_4]\in \mathbf{P}^4\;|\; 2x_0x_4 = x_1^2 + x_2^2 \} \subseteq \mathbf{P}^4. $$
Then there are three distinct $\Ga^3$-structures on $\mathcal{Q}_0^3$. They are given by the following representations of $\Ga^3$:

$$
\begin{array}{l}\rho_1(a_1,a_2,a_3)=
\begin{pmatrix}
1 & 0 & 0 & 0 & 0 \\
a_1 & 1 & 0 & 0 & 0 \\
a_2 & 0 & 1 & 0 & 0 \\
a_3 & 0 & 0 & 1 & 0 \\
\frac{1}{2}(a_1^2+a_2^2) & a_1 & a_2 & 0 & 1
\end{pmatrix}, \\
\rho_2(a_1,a_2,a_3)=
\begin{pmatrix}
1 & 0 & 0 & 0 & 0 \\
a_1 & 1 & 0 & 0 & 0 \\
a_2 & 0 & 1 & 0 & 0 \\
\frac{1}{2}a_2^2+a_3 & 0 & a_2 & 1 & 0 \\
\frac{1}{2}(a_1^2+a_2^2) & a_1 & a_2 & 0 & 1
\end{pmatrix},\\
\rho_3(a_1,a_2,a_3)=
\begin{pmatrix}
1 & 0 & 0 & 0 & 0 \\
a_1 & 1 & 0 & 0 & 0 \\
a_2 & 0 & 1 & 0 & 0 \\
\frac{1}{2}(a_1a_2+ia_2^2)+a_3 & \frac{1}{2}a_2 & \frac{1}{2}a_1+ia_2 & 1 & 0 \\
\frac{1}{2}(a_1^2+a_2^2) & a_1 & a_2 & 0 & 1
\end{pmatrix}, \end{array}
$$
where $i$ denotes a square root of $-1$. In particular, in all the cases the induced action on the hyperplane section $\{x_3=0\}\cap \mathcal{Q}_0^3 \cong \mathcal{Q}_2$ coincides with the action described in Theorem \ref{theo:smooth quadric} for $n=2$ and hence the only invariant curves contained in this section are lines.
\end{thm}

\begin{proof}
 It follows by \cite[Proposition 7, Example 4]{AP14} that all possible $\Ga^3$-structures on $\mathcal{Q}_0^3$ are obtained, via exponentiation, from the local algebras $R_\Lambda = \mathbf{C}[S_1,S_2,S_3]/I_\Lambda$ where
 $$I_\Lambda = \langle S_1S_2 - \lambda_{12}S_3,S_1^2-S_2^2-(\lambda_{11}-\lambda_{22})S_3,S_1S_3,S_2S_3,S_3^2 \rangle $$
 and
 $$\Lambda=\begin{pmatrix}\lambda_{11}&\lambda_{12} \\ \lambda_{21}&\lambda_{22}\end{pmatrix} \in \left\{ \begin{pmatrix}0&0 \\ 0&0\end{pmatrix}, \begin{pmatrix}0&0 \\ 0&1\end{pmatrix}, \begin{pmatrix}\frac{i}{2}&\frac{1}{2} \\ \frac{1}{2}&-\frac{i}{2}\end{pmatrix} \right\}. $$
The result follows. \end{proof}

\section{Additive actions on toric varieties}\label{section:toric}

We refer the reader to \cite{CLS} for the general theory of toric varieties.

\begin{defn}[Normalized action] Let $X$ be a projective toric variety of dimension $n\geq 1$ with acting torus $\mathbf{T}\cong \mathbf{G}_m^n$. Suppose that $X$ admits an additive structure. We say that this structure is {\it normalized by the torus} (or just {\it normalized}) if the group $\mathbf{T}$ normalizes $\Ga^n$ in $\operatorname{Aut}(X)$.
\end{defn}

Inspired by the work of Demazure \cite{Dem70}, Arzhantsev and Romaskevich proved the following result in \cite[Theorem 3.6, Theorem 4.1]{AR17}.

\begin{thm}[Arzhantsev \& Romaskevich]\label{theo: toric existence}
	Let $X$ be a complete toric variety with acting torus $\mathbf{T}$. The following conditions are equivalent:
	\begin{enumerate}
		\item There exists an additive action on $X$ normalized by the torus $\mathbf{T}$.
		\item There exists an additive action on $X$.
	\end{enumerate}
	Moreover, any two normalized additive actions on $X$ are isomorphic.
\end{thm}

In the case of projective toric varieties Arzhantsev and Romaskevich give a combinatorial criterion for the existence of an additive structure in terms of the associated (very ample) polytope.

\begin{defn}[Inscribed polytope]\label{defi:inscribed polytope}
	A lattice polytope $P$ is {\it inscribed in a rectagle} if there is a vertex $v_0\in P$ such that
	\begin{enumerate}
		\item the primitive vectors on the edges of $P$ containing $v_0$ form a basis $e_1,\ldots,e_n$ of the lattice $M$;
		\item for every inequality $\langle p, x \rangle \leq a$ on $P$ that corresponds to a facet of $P$ not passing through $v_0$ we have $\langle p,e_i \rangle \geq 0$ for all $i=1,\ldots,n$.
	\end{enumerate}
\end{defn}

The criterion \cite[Theorem 5.2]{AR17}  can be stated as follows.

\begin{thm}[Arzhantsev \& Romaskevich]\label{theo:additive toric varieties} Let $P$ be a very ample polytope and $X_P$ be the corresponding projective toric variety. Then $X_P$ admits an additive action if and only if the polytope $P$ is inscribed in a rectangle. In particular, a toric Fano manifold $X$ admits an additive action if and only if the polytope $P_{-K_X}$ is inscribed in a rectangle.
\end{thm}

The above criterion allows us to classify all possible smooth toric Fano threefolds which are equivariant compactification of $\Ga^3$ (see Proposition \ref{prop:toric case} below). 

\begin{exmp}[del Pezzo surfaces]\label{example:del Pezzo} Using the above criterion we can deduce that the toric del Pezzo surfaces $\mathbf{P}^1\times \mathbf{P}^1$, $\mathbf{P}^2$, $\mathbf{S}_8$ and $\mathbf{S}_7$ are equivariant compactification of $\Ga^2$, but that $\mathbf{S}_6$ is not an equivariant compactification of $\Ga^2$. In particular, since $\mathbf{S}_d$ is obtained as a blow-up of $\mathbf{S}_{d+1}$ for $1\leq d \leq 7$, if follows from Blanchard's lemma that the del Pezzo surface $\mathbf{S}_d$ is not a compactification of $\Ga^2$ for $1\leq d \leq 6$. This can be also deduced from the explicit classification of $\Ga^2$-structures in $\mathbf{P}^2$ given in Proposition \ref{prop:P2} (cf. \cite{DL10,DL15}).
\end{exmp} 
 
\section{Additive actions on flag varieties}\label{section:flags}

The equivariant compactifications of $\Ga^n$ by flag varieties $G/P$ are studied by Arzhantsev in \cite{Arzhantsev1}. Arzhantsev's criterion reads as follows. 

\begin{thm}[Arzhantsev] Let $G$ be a connected semisimple group of adjoint type and $P$ a parabolic subgroup of $G$. Then the flag variety $G/P$ is an equivariant compactification of $\Ga^n$ if and only if for every pair $(G^{(i)},P^{(i)})$, where $G^{(i)}$ is a simple component of $G$ and $P^{(i)}=G^{(i)}\cap P$, on of the following conditions holds:
\begin{enumerate}
 \item The unipotent radical $P_u^{(i)}$ is commutative.
 \item The pair $(G^{(i)},P^{(i)})$ is exceptional\footnote{We say that $(G,P)$ is {\it exceptional} if $\operatorname{Aut}^0(G/P)\neq G$.}.
\end{enumerate}
\end{thm}

\begin{remark}\label{remark: unicity flag}
 It is worth remarking that $\Ga^n$-structures (if there exist) on flag varieties $G/P$ are all isomorphic as long as $G/P\not\cong \mathbf{P}^n$ (see \cite{FH14} or \cite{Dev15} for instance).
\end{remark}

We denote by $\mathbf{Fl}(1,n-1)$ the $(2n-1)$-dimensional flag variety of lines in hyperplanes in $\mathbf{P}^n$. As a subvariety of $\mathbf{P}^n\times\mathbf{P}^n$ in homogeneous coordinates $[x_i],[y_j]$, it is defined by the equation
$$\sum_{i=0}^n x_i y_i=0.$$
\begin{propo}\label{prop:flag varieties}
	For $n\geqslant 2$, $\mathbf{Fl}(1,n-1)$ is not an equivariant compactification of $\Ga^{2n-1}$.
\end{propo}
\begin{proof}
	We use Arzhantsev's criterion. The variety $F=\mathbf{Fl}(1,n-1)$ is isomorphic to $G/P$ where $G=\operatorname{GL}_{n+1}(\mathbf{C})$ and $P$ is the parabolic subgroup
	\begin{equation*}
	\begin{pmatrix}
	\ast & \ast & \cdots & \ast\\
	0 & \vdots & \ddots & \vdots\\
	\vdots & \ast & \cdots & \ast\\
	0  & \cdots & 0 & \ast
	\end{pmatrix}
	\end{equation*}
	The nilpotent radical of $P$ is
	$$P_u=\begin{pmatrix}
	1 & \ast & \cdots &\cdots & \ast & \ast\\
	& 1 &0 &\cdots & 0  & \ast\\
	& &\ddots & \ddots & \vdots & \vdots\\
	& & &\ddots & 0 & \vdots\\
	& & & & 1 & \ast\\
	& & & & & 1
	\end{pmatrix}
	$$
	This is an extension by the group $\Ga^{n-1}$ of the group $\Ga^{n}$, and it is non-commutative.
	We next show that $\operatorname{Aut}(F)=\operatorname{PGL}_{n+1}(\mathbf{C})$, which is not among the exceptional cases (See \cite[p. 784]{Arzhantsev1}). By Blanchard's lemma for proper morphisms with connected fibers, groups actions can descend equivariantly. The variety $F$ possesses two surjective fibrations to $\mathbf{P}^n$ coming from the two factors. Take an element $\varrho\in\operatorname{Aut}(F)$. It descends to two actions on $\mathbf{P}^n$. But $\varrho$ has to preserve flags. Equivalently it preserves any line and its orthogonal at the same time. So $\varrho$ comes from $\operatorname{Aut}(\mathbf{P}^n)=\operatorname{PGL}_{n+1}(\mathbf{C})$. Conversely any element in $\operatorname{Aut}(\mathbf{P}^n)$ defines an automorphism of $F$. The claim hence follows.\footnote{Alternatively, as Ivan Arzhantsev kindly communicated to us, the non-additivity of $\mathbf{Fl}(1,n-1)$ can be deduced from the classification of parabolic subgroups with abelian unipotent radical in \cite{RRS92}: if a parabolic subgroup $P$ of a simple group $G$ has abelian unipotent radical then $P$ is a maximal parabolic subgroup.}
\end{proof}

\section{Kishimoto's classification}\label{section:kishimoto}

The possible compactifications of contractible affine threefolds into smooth Fano threefolds with $B_2=2$ whose log canonical divisors are not nef were studied by Kishimoto in \cite{Kis05}. See also \cite{Mul90} and \cite{Nag17} for related results. Kishimoto's classification can be resumed as follows.

\begin{thm}[Kishimoto]\label{theo:kishimoto}
	Let $U$ be a contractible affine algebraic threefold. Assume that $U$ is embedded into a smooth Fano threefold $X$ with $B_2(X)=2$ and with boundary $X \setminus U = D_1 \cup D_2$. Suppose that the log canonical divisor $K_X+D_1+D_2$ is not nef. Then the complete list of triplets $(X,D_1\cup D_2,U)$, up to permutation of $D_1$ and $D_2$, is given as in \cite[Table 1]{Kis05}. In particular, $X$ belongs to one of the following 16 deformation equivalence classes:
	\begin{enumerate}
		\item $\II_{14}$, the blow-up of $V_5$ along an elliptic curve which is an intersection of two members of $\left|-\frac{1}{2}K_{V_5} \right|$.
		\item $\II_{20}$, the blow-up of $V_5$ along a twisted cubic.
		\item $\II_{22}$, the blow-up of $V_5$ along a conic.
		\item $\II_{23.(b)}$, the blow-up of $\mathcal{Q}_3 \subseteq \mathbf{P}^4$ along the intersection of $A\in |\mathcal{O}_{\mathcal{Q}_3}(1)|$ and $B\in |\mathcal{O}_{\mathcal{Q}_3}(2)|$ such that $A$ is not smooth.
		\item $\II_{24}$, a divisor on $\mathbf{P}^2 \times \mathbf{P}^2$ of bidegree $(1,2)$.
		\item $\II_{26}$, the blow-up of $V_5$ along a line.
		\item $\II_{27}$, the blow-up of $\mathbf{P}^3$ along a twisted cubic.
		\item $\II_{28}$, the blow-up of $\mathbf{P}^3$ along a plane cubic.
		\item $\II_{29}$, the blow-up of $\mathcal{Q}_3 \subseteq \mathbf{P}^4$ along a conic.
		\item $\II_{30}$, the blow-up of $\mathbf{P}^3$ along a conic or, equivalently, the blow-up of $\mathcal{Q}_3 \subseteq \mathbf{P}^4$ at a point $p$.
		\item $\II_{31}$, the blow-up of $\mathcal{Q}_3 \subseteq \mathbf{P}^4$ along a line.
		\item $\II_{32}$, a divisor on $\mathbf{P}^2 \times \mathbf{P}^2$ of bidegree $(1,1)$ or, equivalently, the projective bundle $\mathbf{P}(T_{\mathbf{P}^2})$.
		\item $\II_{33}$, the blow-up of $\mathbf{P}^3$ along a line or, equivalently, the projective bundle $\mathbf{P}(\mathcal{O}_{\mathbf{P}^1}^{\oplus 2}\oplus \mathcal{O}_{\mathbf{P}^1}(1))$.
		\item $\II_{34} \cong \mathbf{P}^1 \times \mathbf{P}^2$.
		\item $\II_{35}$, the blow-up of $\mathbf{P}^3$ at a point $p$ or, equivalently, the projective bundle $\mathbf{P}(\mathcal{O}_{\mathbf{P}^2}\oplus \mathcal{O}_{\mathbf{P}^2}(1))$.
		\item $\II_{36}\cong \mathbf{P}(\mathcal{O}_{\mathbf{P}^2}\oplus \mathcal{O}_{\mathbf{P}^2}(2))$.
	\end{enumerate}
	Conversely, there is a Fano threefold $X$ belonging to each of these classes such that $X$ is the compactification of a contractible affine algebraic threefold and such that the associated log canonical divisor is not nef.
\end{thm}

\begin{remark}[Topology]
 Kishimoto's classification relies on the fact that the affine space $\mathbf{A}^3$ is contractible (cf. \cite[Corollary 2.1]{Kis05}), which is used several times in \cite{Kis05} for computing topological Euler characteristics. This explains why the varieties considered in this article are defined over the field of complex numbers.
\end{remark}

\section{Additive Fano threefolds}\label{section:Fano 3folds}

From now on we will consider smooth Fano threefolds which are equivariant compactifications of $\mathbf{G}_a^3$.

\subsection{Toric Fano threefolds}\label{section:Toric Fano 3folds}

The aim of this section is to treat the toric case. We may refer the reader to \cite{CLS} for the general theory of toric varieties. 

Toric Fano threefolds where classified by Baryrev in \cite{Bat82} and by Watanabe and Watanabe in \cite{WW82}. Their classification reads as follows.

\begin{thm}[Batyrev, Watanabe \& Watanabe]\label{theo:toric fano 3folds} Let $X$ be a smooth toric Fano threefold. Then $X$ is isomorphic to one of the following toric varieties.
\begin{enumerate}
 \item $\mathbf{P}^3$.
 \item $\II_{33}$, the blow-up of $\mathbf{P}^3$ along a line or, equivalently, the projective bundle $\mathbf{P}(\mathcal{O}_{\mathbf{P}^1}^{\oplus 2}\oplus \mathcal{O}_{\mathbf{P}^1}(1))$.
 \item $\II_{34} \cong \mathbf{P}^1 \times \mathbf{P}^2$.
 \item $\II_{35}$, the blow-up of $\mathbf{P}^3$ at a point $p$ or, equivalently, the projective bundle $\mathbf{P}(\mathcal{O}_{\mathbf{P}^2}\oplus \mathcal{O}_{\mathbf{P}^2}(1))$.
 \item $\II_{36}\cong \mathbf{P}(\mathcal{O}_{\mathbf{P}^2}\oplus \mathcal{O}_{\mathbf{P}^2}(2))$.
 \item $\III_{25}$, the blow-up of $\mathbf{P}^3$ along two disjoint lines or, equivalently, the projective bundle $\mathbf{P}(\mathcal{O}_{\mathbf{P}^1\times \mathbf{P}^1}(1,0)\oplus \mathcal{O}_{\mathbf{P}^1\times \mathbf{P}^1}(0,1))$.
 \item $\III_{26}$, the blow-up of $\mathbf{P}^3$ along the disjoint union of a point and a line.
 \item $\III_{27} \cong \mathbf{P}^1 \times \mathbf{P}^1 \times \mathbf{P}^1$.
 \item $\III_{28} \cong \mathbf{P}^1 \times \mathbf{F}_1$.
 \item $\III_{29}$, the blow-up of $\mathbf{P}(\mathcal{O}_{\mathbf{P}^2}\oplus \mathcal{O}_{\mathbf{P}^2}(1))$ with center a line on the exceptional divisor of the blow-up $\mathbf{P}(\mathcal{O}_{\mathbf{P}^2}\oplus \mathcal{O}_{\mathbf{P}^2}(1))\to \mathbf{P}^3$.
 \item $\III_{30}$, the blow-up of $\mathbf{P}(\mathcal{O}_{\mathbf{P}^2}\oplus \mathcal{O}_{\mathbf{P}^2}(1))$ with center the strict transform of a line passing through the center of the blow-up $\mathbf{P}(\mathcal{O}_{\mathbf{P}^2}\oplus \mathcal{O}_{\mathbf{P}^2}(1))\to \mathbf{P}^3$.
 \item $\III_{31}\cong \mathbf{P}(\mathcal{O}_{\mathbf{P}^1\times \mathbf{P}^1}\oplus \mathcal{O}_{\mathbf{P}^1\times \mathbf{P}^1}(1,1))$.
 \item $\IV_{9}$, the blow-up of $\mathbf{P}(\mathcal{O}_{\mathbf{P}^1\times \mathbf{P}^1}(1,0)\oplus \mathcal{O}_{\mathbf{P}^1\times \mathbf{P}^1}(0,1))$ with center an exceptional line of the blow-up $\mathbf{P}(\mathcal{O}_{\mathbf{P}^1\times \mathbf{P}^1}(1,0)\oplus \mathcal{O}_{\mathbf{P}^1\times \mathbf{P}^1}(0,1)) \to \mathbf{P}^3$.
 \item $\IV_{10} \cong \mathbf{P}^1 \times \mathbf{S}_7$.
 \item $\IV_{11}$, the blow-up of $\mathbf{P}^1 \times \mathbf{F}_1$ with center $\{t\}\times \{E\}$, where $t\in \mathbf{P}^1$ and $E\subseteq \mathbf{F}_1$ is an exceptional curve of the first kind.
 \item $\IV_{12}$, the blow-up of $\mathbf{P}(\mathcal{O}_{\mathbf{P}^1}^{\oplus 2}\oplus \mathcal{O}_{\mathbf{P}^1}(1))$ with center two exceptional lines of the blow-up $\mathbf{P}(\mathcal{O}_{\mathbf{P}^1}^{\oplus 2}\oplus \mathcal{O}_{\mathbf{P}^1}(1))\to \mathbf{P}^3$.
 \item $\V_{2}$, the blow-up of $\mathbf{P}(\mathcal{O}_{\mathbf{P}^1\times \mathbf{P}^1}(1,0)\oplus \mathcal{O}_{\mathbf{P}^1\times \mathbf{P}^1}(0,1))$ with center two exceptional lines $L_1$ and $L_2$ of the blow-up $$\varepsilon:\mathbf{P}(\mathcal{O}_{\mathbf{P}^1\times \mathbf{P}^1}(1,0)\oplus \mathcal{O}_{\mathbf{P}^1\times \mathbf{P}^1}(0,1))\to \mathbf{P}^3$$
 such that both $L_1$ and $L_2$ lie on the same irreducible component of the exceptional set $\operatorname{Exc}(\varepsilon)$.
 \item $\V_{3}\cong \mathbf{P}^1 \times \mathbf{S}_6$.
\end{enumerate}
\end{thm}

Theorem \ref{theo:additive toric varieties} allows us determine which of the varieties in Theorem \ref{theo:toric fano 3folds} are equivariant compactification of $\Ga^3$. 

\begin{propo}\label{prop:toric case}
Let $X$ be a smooth toric Fano threefold. Suppose that $X$ admits a $\Ga^3$-structure. Then $X$ is isomorphic to one of the following.
\begin{enumerate}
 \item $\mathbf{P}^3$.
 \item $\II_{33}$, the blow-up of $\mathbf{P}^3$ with center a line.
 \item $\II_{34}\cong\mathbf{P}^1 \times \mathbf{P}^2$.
 \item $\II_{35}\cong\mathbf{P}(\mathcal{O}_{\mathbf{P}^2}\oplus \mathcal{O}_{\mathbf{P}^2}(1))$, the blow-up of $\mathbf{P}^3$ with center a point.
 \item $\II_{36}\cong\mathbf{P}(\mathcal{O}_{\mathbf{P}^2}\oplus \mathcal{O}_{\mathbf{P}^2}(2))$.
 \item $\III_{26}$, the blow-up of $\mathbf{P}^3$ with center a disjoint union of a point and a line.
 \item $\III_{27}\cong\mathbf{P}^1 \times \mathbf{P}^1 \times \mathbf{P}^1.$
 \item $\III_{28}\cong\mathbf{P}^1\times \mathbf{F}_1$. 
 \item $\III_{29}$, the blow-up of $\mathbf{P}(\mathcal{O}_{\mathbf{P}^2}\oplus \mathcal{O}_{\mathbf{P}^2}(1))$ with center a line on the exceptional divisor of the blow-up $\mathbf{P}(\mathcal{O}_{\mathbf{P}^2}\oplus \mathcal{O}_{\mathbf{P}^2}(1))\to \mathbf{P}^3$.
 \item $\III_{30}$, the blow-up of $\mathbf{P}(\mathcal{O}_{\mathbf{P}^2}\oplus \mathcal{O}_{\mathbf{P}^2}(1))$ with center the strict transform of a line passing through the center of the blow-up $\mathbf{P}(\mathcal{O}_{\mathbf{P}^2}\oplus \mathcal{O}_{\mathbf{P}^2}(1))\to \mathbf{P}^3$.
 \item $\III_{31}\cong \mathbf{P}(\mathcal{O}_{\mathbf{P}^1\times \mathbf{P}^1}\oplus \mathcal{O}_{\mathbf{P}^1\times \mathbf{P}^1}(1,1))$.
 \item $\IV_{10} \cong \mathbf{P}^1 \times \mathbf{S}_7$.
 \item $\IV_{11}$, the blow-up of $\mathbf{P}^1 \times \mathbf{F}_1$ with center $\{t\}\times \{E\}$, where $t\in \mathbf{P}^1$ and $E\subseteq \mathbf{F}_1$ is an exceptional curve of the first kind.
 \item $\IV_{12}$, the blow-up of $\mathbf{P}(\mathcal{O}_{\mathbf{P}^1}^{\oplus 2}\oplus \mathcal{O}_{\mathbf{P}^1}(1))$ with center two exceptional lines of the blow-up $\mathbf{P}(\mathcal{O}_{\mathbf{P}^1}^{\oplus 2}\oplus \mathcal{O}_{\mathbf{P}^1}(1))\to \mathbf{P}^3$.
\end{enumerate}
Conversely, every toric Fano threefold in the above list admits a $\Ga^3$-structure.
\end{propo}

\begin{proof}
By Blanchard's lemma, the fact that the varieties $\III_{26}$, $\III_{29}$, $\IV_{10}, \IV_{11}$ and $\IV_{12}$ are equivariant compactifications of $\Ga^3$ implies that all the other varieties listed above have the same property as well, since they are obtained as blow-downs of these three varieties by \cite{MM81} (cf. family tree in \S 1). The fact that $\III_{26}$ is additive follows from Corollary \ref{coro:P3} and the fact that $\IV_{10}\cong \mathbb{P}^1\times \mathbf{S}_7$ is additive follows from Example \ref{example:del Pezzo}. We refer the reader to Appendix \ref{appendix}, where we verify that $\III_{29}, \IV_{11}$ and $\IV_{12}$ are equivariant compactifications of $\Ga^3$.

Similarly, the fact that $\III_{25}$ is not an equivariant compactification of $\Ga^3$, which follows from from Corollary \ref{coro:P3}, implies that $\IV_9$ and $\V_2$ neither are, since they are obtained as blow-ups of this variety by \cite{MM81}. Finally, $\V_3 \cong \mathbf{P}^1 \times \mathbf{S}_6$ is not an equivariant compactifications of $\Ga^3$ by Lemma \ref{lemma:product actions} (cf. Example \ref{example:del Pezzo}).
\end{proof}

By comparing the varieties appearing in Proposition \ref{prop:toric case} and the list of primitive Fano threefolds in \cite{MM81} we observe the following consequence of the classification of additive toric Fano threefolds.

\begin{cor}\label{primitive toric}
All primitive smooth toric Fano threefolds are equivariant compactifications of $\Ga^3$.
\end{cor} 

\subsection{Fano threefolds with $B_2=1$} It is known after the work of Iskovskikh \cite{Isk77,Isk78,Isk79} and Shokurov \cite{Sho79} that there are 17 families of smooth Fano threefolds with Picard number 1. The classification of smooth additive threefolds with Picard number 1 due to Hassett and Tschinkel (see Theorem \ref{theo:HT}) states that only two among them are additive: $\mathbf{P}^3$ and the smooth quadric hypersurface $\mathcal{Q}_3\subseteq \mathbf{P}^4$.

\subsection{Fano threefolds with $B_2=2$} All the 36 families of smooth Fano threefolds with Picard number 2 were listed by Mori and Mukai in \cite[Table 2]{MM81}. Among these 36 families, Kishimoto proved that only 16 of them admits a member which is the compactification of the affine space $\mathbf{A}^3$ in such a way the associated log canonical divisor is not nef (see \S \ref{section:kishimoto}). In this section we prove that only 7 among these 16 families have members which are equivariant compactifications of $\Ga^3$ (see Proposition \ref{prop:B_2=2} below).

The following result is a direct consequence of Kishimoto's classification.

\begin{lemma}
 The Fano threefolds $\II_{1},\II_{2},\II_{3},\II_{4}, \II_{5}$, $\II_{6}$, $\II_{7}$, $\II_{8}$, $\II_{9}$, $\II_{10}$, $\II_{11}$, $\II_{12}$, $\II_{13}$, $\II_{15}$, $\II_{16}$, $\II_{17}$, $\II_{18}$, $\II_{19}$, $\II_{21}$, $\II_{23.(a)}$ and $\II_{25}$ {\bf are not} equivariant compactifications of $\Ga^3$.
\end{lemma}
\begin{proof}
On one hand, it follows from Theorem \ref{theo:canonical divisor} that if $X$ is an equivariant compactification of $\Ga^3$ with Picard number two and boundary divisors $D_1, D_2$, then the log canonical divisor $K_X+D_1+D_2$ is not nef. 

On the other hand, compactifications of the affine space $\mathbf{A}^3$ into smooth Fano threefolds with Picard number two such that the corresponding log canonical divisor is not nef were classified by Kishimoto in Theorem \ref{theo:kishimoto}. The listed varieties do not appear in Kishimoto's classification and therefore they are not equivariant compactifications of $\Ga^3$.
\end{proof}

\begin{lemma}
 The Fano threefolds $\II_{14},\II_{20},\II_{22}$ and $\II_{26}$ {\bf are not} equivariant compactifications of $\Ga^3$.
\end{lemma}
\begin{proof}
 It follows from Mori and Mukai classification \cite[Table 2]{MM81} that all these varieties are obtained as the blow-up along a curve inside the Fano threefold $V_5$ of Picard number one. If one of these varieties is an equivariant compactification of $\Ga^3$ then we have by Blanchard's lemma that $V_5$ would be such a compactification as well, which contradicts Theorem \ref{theo:HT} above.
\end{proof}

\begin{lemma}
 The Fano threefold $\II_{32}$ {\bf is not} an equivariant compactification of $\Ga^3$.
\end{lemma}
\begin{proof}
 The Fano threefold $\II_{32}$ corresponds to a divisor on $\mathbf{P}^2 \times \mathbf{P}^2$ of bidegree $(1,1)$ or, equivalently, the projective bundle $\mathbf{P}(T_{\mathbf{P}^2})$. The conclusion follows from Proposition \ref{prop:flag varieties} above by considering $n=2$.
\end{proof}

\begin{lemma}
 The Fano threefolds $\II_{23.(b)}$ and $\II_{29}$ {\bf are not} equivariant compactifications of $\Ga^3$. The Fano threefold $\II_{31}$ {\bf is} an equivariant compactification of $\Ga^3$.
\end{lemma}
\begin{proof}
 All these varieties are obtained as the blow-up of the smooth quadric $\mathcal{Q}_3 \subseteq \mathbf{P}^4$ along a curve (see Mori and Mukai classification \cite[Table 2]{MM81}). As pointed out in Corollary \ref{cor:smooth quartic 3fold}, the only curves which are invariant by the unique action in $\mathcal{Q}_3 \subseteq \mathbf{P}^4$ are lines contained on the boundary divisor $\mathcal{Q}_0^2 \subseteq \mathcal{Q}_3$ corresponding the ruling passing through the singular point of $\mathcal{Q}_0^2$. The conclusion follows from Corollary \ref{co:equivariant blow-up} above since $\II_{31}$ is the only variety among these three that is obtained as the blow-up of an invariant curve (i.e. a line passing through the singular point of the boundary divisor).
\end{proof}

\begin{lemma}
 The Fano threefold $\II_{24}$ {\bf is not} an equivariant compactification of $\Ga^3$.
\end{lemma}
\begin{proof}
 A member $X$ of the family $\II_{24}$ corresponds to a divisor on $\mathbf{P}^2 \times \mathbf{P}^2$ of bidegree $(1,2)$. The Fano threefold $X$ admits therefore two different extremal contractions of fiber type $\varphi_i: X \to \mathbf{P}^2$ ($i=1,2$), which are induced by the projections onto each of the factors of $\mathbf{P}^2 \times \mathbf{P}^2$. Let us assume that $X$ is an equivariant compactification of $\Ga^3$ with boundary divisors $D_1,D_2$ and let $\Delta:=D_1+D_2$. Since $X$ has no divisorial contractions it follows from \cite[Lemma 2.2]{Kis05} that $\Delta$ is ample.
 
 On one hand, it follows from the adjunction formula that $-K_X = 2H_1+H_2$, where $H_1=c_1(\varphi_1^*\mathcal{O}_{\mathbf{P}^2}(1))$ and $H_2=c_1(\varphi_2^*\mathcal{O}_{\mathbf{P}^2}(1))$. On the other hand, since $\Delta$ is ample it follows from the proof of \cite[Lemma 5.6]{Kis05} that $D_2\sim H_1$ and $D_1 \sim aH_1+H_2$ with $a\in \mathbf{Z}$. Therefore we can write $-K_X=D_1+(2-a)D_2$ with $a\in \mathbf{Z}$, which contradicts Theorem \ref{theo:canonical divisor}.
\end{proof}

\begin{lemma}
 The Fano threefold $\II_{27}$ {\bf is not} an equivariant compactification of $\Ga^3$. The Fano threefolds $\II_{28}$ and $\II_{30}$ {\bf are} equivariant compactifications of $\Ga^3$.
\end{lemma}
\begin{proof}
 All these varieties are obtained as the blow-up of $\mathbf{P}^3$ along a curve (see Mori and Mukai classification \cite[Table 2]{MM81}). It follows from Corollary \ref{coro:P3} that the only curves that are invariant for some of the four possible $\Ga^3-$structures in $\mathbf{P}^3$ are curves inside a plane $\mathbf{P}^2\subseteq \mathbf{P}^3$. Whereas the variety $\II_{27}$ is obtained by blowing-up a twisted cubic, the varieties $\II_{28}$ and $\II_{30}$ are obtained by blowing-up planar curves in $\mathbf{P}^3$. The conclusion follows from Corollary \ref{co:equivariant blow-up}.
\end{proof}

\begin{lemma}
 The toric Fano threefolds $\II_{33}, \II_{34}, \II_{35}$ and $\II_{36}$ {\bf are} equivariant compactifications of $\Ga^3$.
\end{lemma}
\begin{proof}
 This is a particular case of Proposition \ref{prop:toric case}.
\end{proof}

Our analysis is summarized in the following list.

\begin{propo}\label{prop:B_2=2}
Let $X$ be a smooth Fano threefold with $B_2(X)=2$. Suppose that $X$ admits a $\Ga^3$-structure. Then $X$ is isomorphic to one of the following.
\begin{enumerate}
 \item $\II_{28}$, the blow-up of $\mathbf{P}^3$ along a plane cubic.
 \item $\II_{30}$, the blow-up of $\mathbf{P}^3$ along a conic or, equivalently, the blow-up of $\mathcal{Q}_3 \subseteq \mathbf{P}^4$ at a point $p$.
 \item $\II_{31}$, the blow-up of $\mathcal{Q}_3 \subseteq \mathbf{P}^4$ along a line.
 \item $\II_{33}$, the blow-up of $\mathbf{P}^3$ along a line or, equivalently, the projective bundle $\mathbf{P}(\mathcal{O}_{\mathbf{P}^1}^{\oplus 2}\oplus \mathcal{O}_{\mathbf{P}^1}(1))$.
 \item $\II_{34} \cong \mathbf{P}^1 \times \mathbf{P}^2$.
 \item $\II_{35}$, the blow-up of $\mathbf{P}^3$ at a point $p$ or, equivalently, the projective bundle $\mathbf{P}(\mathcal{O}_{\mathbf{P}^2}\oplus \mathcal{O}_{\mathbf{P}^2}(1))$.
 \item $\II_{36}\cong \mathbf{P}(\mathcal{O}_{\mathbf{P}^2}\oplus \mathcal{O}_{\mathbf{P}^2}(2))$.
\end{enumerate}
Conversely, every Fano threefold in the above list admits a $\Ga^3$-structure.
\end{propo}

\subsection{Fano threefolds with $B_2=3$} All the 31 families of smooth Fano threefolds with Picard number 3 were listed by Mori and Mukai in \cite[Table 3]{MM81}. In this section we prove that only 7 among these 31 families have members which are equivariant compactifications of $\Ga^3$ (see Proposition \ref{prop:B_2=3} below).

\begin{lemma}
	The primitive Fano threefold $\III_{1}$ {\bf is not} an equivariant compactification of $\Ga^3$.
\end{lemma}
\begin{proof}
	It follows from \cite[p. 71]{AB92} that $\III_{1}$ is not rational, hence it can not be the equivariant compactification of $\Ga^3$. See also \cite[$\S$12.4]{IP99}.
\end{proof}

\begin{lemma}
	The primitive Fano threefold $\III_{2}$ {\bf is not} an equivariant compactification of $\Ga^3$.
\end{lemma}
\begin{proof}
	Let $\pi: P=\mathbf{P}(\mathcal{O}_{\mathbf{P}^1\times \mathbf{P}^1}\oplus \mathcal{O}_{\mathbf{P}^1\times \mathbf{P}^1}(-1,-1)^{\oplus 2}) \to \mathbf{P}^1\times \mathbf{P}^1$ with tautological line bundle $\mathcal{O}_{P}(1)$. The primitive Fano threefold $\III_{2}$ is given by a smooth member $X$ of the linear system $|\mathcal{O}_P(2)\otimes \pi^* \mathcal{O}_{\mathbf{P}^1\times \mathbf{P}^1}(2,3)|$ and it is endowed with a natural projection $\rho:X \to \mathbf{P}^1\times \mathbf{P}^1$ induced by $\pi$. 
	
	It is well known that if $E$ is a rank $r$ vector bundle over a smooth projective variety $Z$ and $\pi:\mathbf{P}(E)\to Z$, then we have the canonical bundle formula $\omega_{\mathbf{P}(E)} \cong \mathcal{O}_{\mathbf{P}(E)}(-r)\otimes \pi^*(\omega_Z \otimes \det(E))$. In our case we get that $\omega_P\cong \mathcal{O}_P(-3)\otimes \pi^*\mathcal{O}_{\mathbf{P}^1\times \mathbf{P}^1}(-4,-4)$ and hence the adjunction formula leads to $\omega_X^{\vee} \cong \mathcal{O}_X(1)\otimes \rho^*\mathcal{O}_{\mathbf{P}^1\times \mathbf{P}^1}(1,2)$, which is a primitive vector in the lattice $\mathbf{NS}(X)\cong \mathbf{Z}^3$ generated by the first Chern classes $H_1=c_1(\rho^*\mathcal{O}_{\mathbf{P}^1\times \mathbf{P}^1}(1,0)), H_2=c_1(\rho^*\mathcal{O}_{\mathbf{P}^1\times \mathbf{P}^1}(0,1))$ and $H_3=c_1(\mathcal{O}_X(1))$.
	
	Let us assume by contradiction that $X$ is the equivariant compactification of $\Ga^3$ and let $\Delta = D_1+D_2+D_3$ be the boundary divisor. It follows from Theorem \ref{theo:canonical divisor} that $-K_X-2\Delta$ is an effective divisor on $X$. We note that $-(K_X+2\Delta)$ is not zero (otherwise we would deduce that $-K_X = 2\Delta$ is not a primitive vector in $\mathbf{NS}(X)$). In particular, since $\NE(X)$ is a rational polyhedral cone, there is an irreducible curve $C\subseteq X$ such that $-(K_X+2\Delta)\cdot C > 0$. In other words, $K_X+2\Delta$ is not nef.
	
	It follows by \cite{Fuj87} that if $\Delta$ is ample and $K_X+2\Delta$ is not nef, then we have that $X$ is either isomorphic to $\mathcal{Q}_3\subseteq \mathbf{P}^4$ or $X$ is a $\mathbf{P}^2$-bundle over $\mathbf{P}^1$ and hence $X\cong \mathbf{P}(E)$ for some rank 3 vector bundle $E$ on $\mathbf{P}^1$, by \cite[Remark 9]{AR14}. In both cases we have that $B_2(X)\neq 3$, a contradiction.
	
	Let us suppose that $\Delta$ is not ample. We will follow a similar strategy as in \cite[Lemma 2.2]{Kis05}.
	
	If $\Delta$ is nef but not ample then there exists an irreducible curve $C\subseteq X$ such that $\Delta\cdot C = 0$. Since $\NE(X)$ is a rational polyhedral cone there is an extremal ray $R=\mathbf{R}_{\geq 0}[C_R]\subseteq \NE(X)$ such that $\Delta \cdot R = 0$. Let $\varphi_R:X\to X_R$ the associated extremal contraction with irreducible exceptional locus $\Exc(\varphi_R)$. Since $\Ga^3$ contains no complete curves, we must have that $C_R\subseteq D_1\cup D_2\cup D_3$ and hence $\varphi_R$ is a divisorial contraction such that $\Exc(\varphi_R)=D_i$ for some $i=1,2,3$. On one hand, we note that $\varphi_R:X\to X_R$ is not of type $E1$\footnote{See \cite[Theorem 3.3]{Mor82} for the classification of divisorial contractions on smooth threefolds. We keep the same notation as in \cite[Theorem 1.32]{KM98}.}, as we would have that $X_R$ is an equivariant compactification of $\Ga^3$ with boundary $\varphi_R(D_1)\cup \varphi_R(D_2)\cup \varphi_R(D_3)$ which is not of pure codimension 1, a contradiction with the fact that $\Ga^3$ is affine (cf. Proposition \ref{prop:boundary divisors}). On the other hand, Mori and Mukai proved in \cite[$\S$ 8.5]{MM83} that primitive Fano threefolds with $B_2(X)=3$ have no extremal contractions of type $E2,E3,E4$ or $E5$. Hence $\varphi_R$ cannot be divisorial, a contradiction.
	
	If $\Delta$ is not nef, then there exists an extremal ray $R=\mathbf{R}_{\geq 0}[C_R]\subseteq \NE(X)$ such that $\Delta \cdot R < 0$. Therefore, $\varphi_R$ is divisorial and $\Exc(\varphi_R)=D_i$ for some $i=1,2,3$. The same argument as above leads us to a contradiction.
\end{proof}

\begin{lemma}
 The Fano threefolds $\III_3$, $\III_5$, $\III_7$, $\III_{11}$, $\III_{12}$, $\III_{15}$, $\III_{17}$, $\III_{21}$, $\III_{22}$ and $\III_{24}$ {\bf are not} equivariant compactifications of $\Ga^3$.
\end{lemma}
\begin{proof} On one hand, all these varieties are obtained as the blow-up of $\II_{34}\cong \mathbf{P}^1\times \mathbf{P}^2$ along a curve (see Mori and Mukai classification \cite[Table 3]{MM81}). On the other hand, it follows by the classification of all possible $\Ga^3-$structures on $\mathbf{P}^1\times \mathbf{P}^2$ given in Lemma \ref{lemma:product actions} that the only additive Fano threefolds obtained as the blow-up of $\mathbf{P}^1\times \mathbf{P}^2$ along a curve are toric\footnote{They correspond to the varieties $\III_{26}$ and $\III_{28}$ (cf. Lemma \ref{lemma:B_2=3 toric} below).}. The conclusion follows from the classification of smooth toric Fano threefolds (see Theorem \ref{theo:toric fano 3folds}).
\end{proof}

\begin{lemma}
 The Fano threefolds $\III_4$, $\III_6$, $\III_8$, $\III_{13}$, $\III_{16}$, $\III_{18}$, $\III_{20}$ {\bf are not} equivariant compactifications of $\Ga^3$.
\end{lemma}
\begin{proof}
 It follows from Mori and Mukai classification \cite[Table 3]{MM81} that all these varieties are obtained as the blow-up along a curve on Fano threefolds which do not belong to the families listed in Proposition \ref{prop:B_2=2}. If one of these varieties is an equivariant compactification of $\Ga^3$ then Blanchard's lemma and Corollary \ref{co:equivariant blow-up} would lead us to a contradiction.
\end{proof}

\begin{lemma}
	The Fano threefold $\III_{9}$ {\bf is not} an equivariant compactification of $\Ga^3$.
\end{lemma}
\begin{proof}
A member $X$ of the family $\III_9$ is given by the blow-up of the cone $W\subseteq \mathbf{P}^6$ over the Veronese surface $S\subseteq \mathbf{P}^5$ with center the disjoint union of the vertex of $W$ and a quadric curve in $S\cong \mathbf{P}^2$. Let us note that $W$ is isomorphic to the weighted projective space $\mathbf{P}(1,1,1,2)$ and that the blow-up of the vertex $\psi:Y\to W$ is a resolution of singularities, where $Y\cong \mathbf{P}(\mathcal{O}_{\mathbf{P}^2}\oplus \mathcal{O}_{\mathbf{P}^2}(2))$ and $\psi$ corresponds via this isomorphism to the divisorial contraction sending the section $S_0\cong \mathbf{P}^2$ of the structural morphism $\pi:Y\to \mathbf{P}^2$ associated to the quotient $\mathcal{O}_{\mathbf{P}^2}\oplus \mathcal{O}_{\mathbf{P}^2}(2)\to \mathcal{O}_{\mathbf{P}^2}$ into the vertex of $W$. Let $S_\infty \cong \mathbf{P}^2$ be the section of $\pi:Y\to \mathbf{P}^2$ associated to the quotient $\mathcal{O}_{\mathbf{P}^2}\oplus \mathcal{O}_{\mathbf{P}^2}(2)\to \mathcal{O}_{\mathbf{P}^2}(2)$ and let $C\subseteq S_\infty$ be a quartic curve. Then $\sigma:X\to Y$ is given by the blow-up of $Y$ along $C$, and the image $\psi(S_\infty)$ in $W$ is isomorphic to the Veronese surface $S\subseteq \mathbf{P}^5$.

Let us assume that $X$ is an equivariant compactification of $\Ga^3$. On one hand, given a $\Ga^3-$structure on $X$, Blanchard's lemma implies that there is a unique $\Ga^3-$structure on $Y$ such that $\sigma:X\to Y$ is an equivariant morphism and that $Y$ an equivariant compactification of $\Ga^3$. On the other hand, Kishimoto's classification \cite[Table 1]{Kis05} shows that there is a unique way to compactify $\mathbf{A}^3$ into $\mathbf{P}(\mathcal{O}_{\mathbf{P}^2}\oplus \mathcal{O}_{\mathbf{P}^2}(2))$ with not nef log canonical divisor. In this case, the boundary divisors are given by $D_1\cong \mathbf{P}^2$, the section $S_0\cong \mathbf{P}^2$ of $\pi:Y\to \mathbf{P}^2$, and $D_2 \cong \mathbf{F}_2$ obtained as the pullback by $\pi$ of a line in $\mathbf{P}^2$. In particular, the intersection curve $S_\infty \cap D_2$ is a line, and therefore quartic curves in $S_\infty$ are not invariant by the $\Ga^3-$action, a contradiction.
\end{proof}

\begin{lemma}
The Fano threefolds $\III_{10}$, $\III_{15}$ and $\III_{19}$ {\bf are not} equivariant compactifications of $\Ga^3$.
\end{lemma}
\begin{proof}
The varieties $\III_{10}$ and $\III_{15}$ are obtained as the blow-up of the smooth quadric $\mathcal{Q}_3 \subseteq \mathbf{P}^4$ along two disjoint curves, while the variety $\III_{19}$ is obtained as the blow-up of $\mathcal{Q}_3 \subseteq \mathbf{P}^4$ with center two points which are not colinear (see Mori and Mukai classification \cite[Table 3]{MM81}). By Corollary \ref{cor:smooth quartic 3fold}, the only curves which are invariant by the unique action in $\mathcal{Q}_3 \subseteq \mathbf{P}^4$ are lines contained on the boundary divisor $\mathcal{Q}_0^2 \subseteq \mathcal{Q}_3$ corresponding the ruling passing through the singular point of $\mathcal{Q}_0^2$. In particular, the invariant curves are not disjoint and the only fixed point by the action is the singular point of $\mathcal{Q}_0^2$. The conclusion follows from Corollary \ref{co:equivariant blow-up}.
\end{proof}

\begin{lemma}
The Fano threefold $\III_{14}$ {\bf is not} an equivariant compactification of $\Ga^3$.
\end{lemma}
\begin{proof}
The variety $\III_{14}$ is obtained as the blow-up of $\mathbf{P}^3$ with center the disjoint union of a cubic curve in a plane $\mathbf{P}^2\subseteq \mathbf{P}^3$ and a point not in this plane (see Mori and Mukai classification \cite[Table 3]{MM81}). It follows from Corollary \ref{coro:P3} that the only invariant (proper) subvarieties for some of the four possible $\Ga^3-$structures in $\mathbf{P}^3$ are contained in a plane $\mathbf{P}^2\subseteq \mathbf{P}^3$. The conclusion follows from Corollary \ref{co:equivariant blow-up}.
\end{proof}

\begin{lemma}\label{lemma:III23}
	There is a Fano threefold in the family $\III_{23}$ which {\bf is} an equivariant compactification of $\Ga^3$.
\end{lemma}
\begin{proof}
	The member $X$ of the family $\III_{23}$ corresponds to the blow-up of $\mathbf{P}(\mathcal{O}_{\mathbf{P}^2}\oplus \mathcal{O}_{\mathbf{P}^2}(1))$ with center a conic passing through the center of the blow-up $\mathbf{P}(\mathcal{O}_{\mathbf{P}^2}\oplus \mathcal{O}_{\mathbf{P}^2}(1))\to \mathbf{P}^3$. A $\Ga^3-$structure induces an action on the exceptional divisor $E\cong \mathbf{P}^2$ on $\II_{35}\cong \mathbf{P}(\mathcal{O}_{\mathbf{P}^2}\oplus \mathcal{O}_{\mathbf{P}^2}(1))$. It suffices to consider the normalized $\Ga^3-$structure on $\mathbf{P}^3$ with fixed locus a plane $\mathbf{P}^2$, and take a conic in this plane with tangent direction given by a fixed point on the exceptional divisor (with respect to the induced action). Then the strict transform of this conic is totally invariant under the lifted $\Ga^3-$action and thus this action lifts again to $X$.
\end{proof}

\begin{lemma}\label{lemma:B_2=3 toric}
 The toric Fano threefold $\III_{25}$ {\bf is not} an equivariant compactification of $\Ga^3$. The toric Fano threefolds $\III_{26}$, $\III_{27}$, $\III_{28}$, $\III_{29}$, $\III_{30}$ and $\III_{31}$ {\bf are} equivariant compactifications of $\Ga^3$.
\end{lemma}
\begin{proof}
 This is a particular case of Proposition \ref{prop:toric case}.
\end{proof}

Our analysis is summarized in the following list.

\begin{propo}\label{prop:B_2=3}
Let $X$ be a smooth Fano threefold with $B_2(X)=3$. Suppose that $X$ admits a $\Ga^3$-structure. Then $X$ is isomorphic to one of the following.
\begin{enumerate}
 \item $\III_{23}$, the blow-up of $\mathbf{P}(\mathcal{O}_{\mathbf{P}^2}\oplus \mathcal{O}_{\mathbf{P}^2}(1))$ with center a conic passing through the center of the blow-up $\mathbf{P}(\mathcal{O}_{\mathbf{P}^2}\oplus \mathcal{O}_{\mathbf{P}^2}(1))\to \mathbf{P}^3$.
 \item $\III_{26}$, the blow-up of $\mathbf{P}^3$ with center a disjoint union of a point and a line.
 \item $\III_{27}\cong\mathbf{P}^1 \times \mathbf{P}^1 \times \mathbf{P}^1.$
 \item $\III_{28}\cong\mathbf{P}^1\times \mathbf{F}_1$. 
 \item $\III_{29}$, the blow-up of $\mathbf{P}(\mathcal{O}_{\mathbf{P}^2}\oplus \mathcal{O}_{\mathbf{P}^2}(1))$ with center a line on the exceptional divisor of the blow-up $\mathbf{P}(\mathcal{O}_{\mathbf{P}^2}\oplus \mathcal{O}_{\mathbf{P}^2}(1))\to \mathbf{P}^3$.
 \item $\III_{30}$, the blow-up of $\mathbf{P}(\mathcal{O}_{\mathbf{P}^2}\oplus \mathcal{O}_{\mathbf{P}^2}(1))$ with center the strict transform of a line passing through the center of the blow-up $\mathbf{P}(\mathcal{O}_{\mathbf{P}^2}\oplus \mathcal{O}_{\mathbf{P}^2}(1))\to \mathbf{P}^3$.
 \item $\III_{31}\cong \mathbf{P}(\mathcal{O}_{\mathbf{P}^1\times \mathbf{P}^1}\oplus \mathcal{O}_{\mathbf{P}^1\times \mathbf{P}^1}(1,1))$.
\end{enumerate}
Conversely, every Fano threefold in the above listed families\footnote{For $\III_{23}$ this is true only for a particular member of the family (see Lemma \ref{lemma:III23} and its proof).} admits a $\Ga^3$-structure. 
\end{propo}

\subsection{Fano threefolds with $B_2=4$} All the 13 families of smooth Fano threefolds with Picard number 4 were listed by Mori and Mukai in \cite[Table 4]{MM81} and \cite[Addendum]{MM03}. In this section we prove that only 3 among these 13 families have members which are equivariant compactifications of $\Ga^3$ (see Proposition \ref{prop:B_2=4} below).

\begin{lemma}
 The Fano threefolds $\IV_1$, $\IV_3$, $\IV_6$, $\IV_{8}$ and $\IV_{13}$ {\bf are not} equivariant compactifications of $\Ga^3$.
\end{lemma}
\begin{proof} On one hand, all these varieties are obtained as the blow-up of $\mathbf{P}^1\times \mathbf{P}^1 \times \mathbf{P}^1$ along a curve (see Mori and Mukai classification \cite[Table 4]{MM81}). On the other hand, it follows by the classification of all possible $\Ga^3$-structures on $\mathbf{P}^1\times \mathbf{P}^1 \times \mathbf{P}^1$ given in Lemma \ref{lemma:product actions} that the only additive Fano threefold obtained as the blow-up of $\mathbf{P}^1\times \mathbf{P}^1 \times \mathbf{P}^1$ along a curve is toric\footnote{It corresponds to the variety $\IV_{10}$ (cf. Lemma \ref{lemma:B_2=4 toric} below).}. The conclusion follows from the classification of smooth toric Fano threefolds (see Theorem \ref{theo:toric fano 3folds}).
\end{proof}

\begin{lemma}
 The Fano threefolds $\IV_4$, $\IV_5$, $\IV_7$ and $\IV_9$ {\bf are not} equivariant compactifications of $\Ga^3$.
\end{lemma}
\begin{proof}
It follows from Mori and Mukai classification \cite[Table 3]{MM81} that all these varieties are obtained as the blow-up along a curve on Fano threefolds which do not belong to the families listed in Proposition \ref{prop:B_2=3}. If one of these varieties is an equivariant compactification of $\Ga^3$ then Blanchard's lemma and Corollary \ref{co:equivariant blow-up} would lead us to a contradiction.
\end{proof}

\begin{lemma}
 The Fano threefold $\IV_2$ {\bf is not} an equivariant compactification of $\Ga^3$.
\end{lemma}
\begin{proof}
A member $X$ of the family $\III_9$ is given by the blow-up of the cone $Q_0^3 \subseteq \mathbf{P}^4$ over a smooth quadric surface $Q_2\subseteq \mathbf{P}^3$ with center the disjoint union of the vertex of $\mathcal{Q}_0^3$ and an elliptic curve in $\mathcal{Q}_2$. As pointed out in Theorem \ref{theo:degenerate quadric}, the only invariant curves (for all possible actions) contained in the hyperplane section $Q_2 \subseteq \mathcal{Q}_0^3$ are lines. The conclusion follows from Corollary \ref{co:equivariant blow-up}.
\end{proof}

\begin{lemma}\label{lemma:B_2=4 toric}
 The toric Fano threefold $\IV_{9}$ {\bf is not} an equivariant compactification of $\Ga^3$. The toric Fano threefolds $\IV_{10}$, $\IV_{11}$ and $\IV_{12}$ {\bf are} equivariant compactifications of $\Ga^3$.
\end{lemma}
\begin{proof}
 This is a particular case of Proposition \ref{prop:toric case}.
\end{proof}

Our analysis is summarized in the following list.

\begin{propo}\label{prop:B_2=4}
Let $X$ be a smooth Fano threefold with $B_2(X)=4$. Suppose that $X$ admits a $\Ga^3$-structure. Then $X$ is isomorphic to one of the following.
\begin{enumerate}
 \item $\IV_{10} \cong \mathbf{P}^1 \times \mathbf{S}_7$.
 \item $\IV_{11}$, the blow-up of $\mathbf{P}^1 \times \mathbf{F}_1$ with center $\{t\}\times \{E\}$, where $t\in \mathbf{P}^1$ and $E\subseteq \mathbf{F}_1$ is an exceptional curve of the first kind.
 \item $\IV_{12}$, the blow-up of $\mathbf{P}(\mathcal{O}_{\mathbf{P}^1}^{\oplus 2}\oplus \mathcal{O}_{\mathbf{P}^1}(1))$ with center two exceptional lines of the blow-up $\mathbf{P}(\mathcal{O}_{\mathbf{P}^1}^{\oplus 2}\oplus \mathcal{O}_{\mathbf{P}^1}(1))\to \mathbf{P}^3$.
\end{enumerate}
Conversely, every Fano threefold in the above list admits a $\Ga^3$-structure.
\end{propo}

\subsection{Fano threefolds with $B_2=5$} All the 3 families of smooth Fano threefolds with Picard number 5 were listed by Mori and Mukai in \cite[Table 5]{MM81}. In this section we prove that none of these families have a member which is an equivariant compactifications of $\Ga^3$ (see Proposition \ref{prop:B_2=5} below).

\begin{lemma}
 The Fano threefolds $\V_1$ and $\V_2$ {\bf are not} equivariant compactifications of $\Ga^3$.
\end{lemma}
\begin{proof}
 It follows from Mori and Mukai classification \cite[Table 3]{MM81} that all these varieties are obtained as the blow-up along a curve on Fano threefolds which do not belong to the families listed in Proposition \ref{prop:B_2=4}. If one of these varieties is an equivariant compactification of $\Ga^3$ then Blanchard's lemma and Corollary \ref{co:equivariant blow-up} would lead us to a contradiction.
\end{proof}

\begin{lemma}
 The toric Fano threefold $\V_{3}$ {\bf is not} an equivariant compactification of $\Ga^3$.
\end{lemma}
\begin{proof}
 This is a particular case of Proposition \ref{prop:toric case}.
\end{proof}

As a consequence we obtain the following result.

\begin{propo}\label{prop:B_2=5}
Let $X$ be a smooth Fano threefold. Suppose that $X$ admits a $\Ga^3$-structure. Then $B_2(X)\neq 5$.
\end{propo}

\subsection{Fano threefolds with $B_2\geq 6$} All the 5 families of smooth Fano threefolds with Picard number $\geq 6$ were listed by Mori and Mukai in \cite[Table 5]{MM81}. In this section we prove that none of these families have a member which is an equivariant compactifications of $\Ga^3$ (see Proposition \ref{prop:B_2 large} below).

\begin{propo}\label{prop:B_2 large}
Let $X$ be a smooth Fano threefold. Suppose that $X$ admits a $\Ga^3$-structure. Then $B_2(X)\leq 4$.
\end{propo}
\begin{proof}
 By Proposition \ref{prop:B_2=5} it suffices to show that if $X$ is a smooth Fano threefold with $B_2(X)\geq 6$ then $X$ does not admit a $\Ga^3$-structure. It follows by Mori and Mukai classification \cite[Theorem 2, Table 5]{MM81} that such a threefold can be written as a product $X\cong \mathbf{P}^1\times \mathbf{S}_d$, where $\mathbf{S}_d$ is a del Pezzo surface of degree $1\leq d \leq 5$. It is a classical fact that all these del Pezzo surfaces have finite automorphism group (see for instance \cite[Corollary 8.2.40]{Dol12}). In particular, Corollary \ref{co:productauto} gives
 $$\operatorname{Aut}^0(X)\cong \operatorname{Aut}^0(\mathbf{P}^1) \times \operatorname{Aut}^0(\mathbf{S}_d)\cong \operatorname{PGL}_2(\mathbf{C}).$$
 We note that $\Ga^3$ does not embed into $\operatorname{PGL}_2(\mathbf{C})$ since both groups are three-dimensional connected algebraic groups and the former is commutative while the latter is not.
\end{proof}

\section{Further discussion, questions and open problems}\label{section:open problems}

The aim of this section is to discuss some open problems and questions that arise.

\subsection{Automorphism groups}

Smooth Fano threefolds of Picard number one with infinite automorphism group were studied by Prokhorov in \cite{Pro90}: they correspond to $\mathbf{P}^3$, $\mathcal{Q}_3\subseteq \mathbf{P}^4$, $V_5$ and $V_{22}$. See also \cite[Theorem 1.1.2]{KPS18}. In particular, these varieties are precisely the possible compactifications of the affine space $\mathbf{A}^3$ into smooth Fano threefolds with $B_2=1$, by \cite{BM78,Fur86,Fur90,Fur93a,Fur93b,FN89a,FN89b,Muk92,PS88,Pet89,Pet90,Pro91}. 

By looking at the classification of Fano threefolds of Picard number one, Kuznetsov, Prokhorov and Shramov observed in \cite[Corollary 1.1.3]{KPS18} that these varieties are precisely the only Fano threefolds of Picard number one that verify $h^{1,2}(X)=0$ (and hence, they have trivial Intermediate Jacobian).

It is worth remarking that this last observation does not hold in our setting. In fact, all additive Fano threefolds classified in Main Theorem verifies $h^{1,2}(X)=0$ with the exception of $\II_{28}$, i.e., the blow-up of $\mathbf{P}^3$ along a plane cubic. It is however natural to study the automorphism group of the varieties listed in Main Theorem.

\begin{prob}
 Compute the automorphism group $\operatorname{Aut}(X)$ for each of the varieties listed in Main Theorem and understand the inclusion $\mathbf{G}_a^3 \hookrightarrow \operatorname{Aut}^0(X)$.
\end{prob}

\subsection{Higher dimensional toric varieties}

The proof of Proposition \ref{prop:toric case} given in Appendix A below is very algorithmic. It is natural to ask the following.

\begin{prob}
Given the fan and the anticanonical polytope of a smooth toric Fano variety $X$, find an algorithmic way to verify if $X$ is additive or not. For instance, using Macaulay2 \cite{M2}. In particular, determine which of the 124 smooth toric Fano fourfolds classified by Batyrev and Sato in \cite{Bat99,Sat00} are additive.
\end{prob}

\subsection{Higher complexity torus action} The additivity criterion of Arzhantsev and Romaskevich (see Theorem \ref{theo:additive toric varieties}) relies on the connection between $\Ga$-actions and Demazure roots. If instead of considering toric varieties we consider {\it $T$-varieties} of higher complexity, i.e., normal varieties $X$ endowed with an effective action of a torus $\mathbf{T}$ of dimension $\dim(X)-c$, where $c\geq 1$ is the {\it complexity} of the action, then a combinatorial description of such varieties is given by Altmann, Hausen and S\"{u}ss in \cite{AH06,AHS08}. Moreover, locally nilpotent derivations (and hence, $\Ga$-actions) on affine $T$-varieties of complexity-one were studied by Liendo in \cite{Lie10} (cf. \cite{Lie10b,LL16}). The following question arises.

\begin{ques}
 It is possible to extend Arzhantsev-Romaskevich criterion to (not necessarily Fano) projective $T$-varieties of complexity-one? 
\end{ques}

In the Fano case, one possible candidate for being the complexity-one analog of the anticanonical polytope $P$ in the toric case could be the so called {\it divisorial polytope} of $(X,-K_X)$, defined in \cite{IS11} (cf. \cite{Sus14,IS17}). Another candidate could be the Newton-Okounkov body (see \cite{LM09,KK12}) of the anticanonical divisor $\Delta_{Y_\bullet}(-K_X)$ with respect to some flag $Y_\bullet$ of $T$-invariant subvarieties (see \cite{Pet11,IM18}). See also the recent prepublication \cite[\S 2]{BWW18} for another possible point of view.

\subsection{Fano fourfolds with $B_2=1$} Let $X$ a smooth Fano fourfold of Picard number one. In virtue of Theorem \ref{theo:canonical divisor} above, the Fano index of $X$ verifies $i_X\geq 2$. It is known after the work of Kobayashi and Ochiai \cite{KO73} that $i_X\leq 5$ and that $i_X=5$ if and only if $X\cong \mathbf{P}^4$, and that $i_X=4$ if and only if $X\cong \mathcal{Q}_4\subseteq \mathbf{P}^5$. Fano fourfolds with index $i_X=3$ were classified by Fujita \cite{Fuj80,Fuj81,Fuj84} into 5 families and they are called {\it del Pezzo} fourfolds. Fano fourfolds with index $i_X=2$ were classified by Mukai and Wilson \cite{Muk89,Wil87} into 9 families and they are called {\it Fano-Mukai} fourfolds.

\begin{prob}
 Determine which Fano fourfolds with index $i_X\in \{2,3\}$ are additive.
\end{prob}

Prokhorov proved in \cite{Pro94} that only one family of del Pezzo fourfolds can be realized as the compactification of the affine space $\mathbf{A}^4$: a section of the Grassmannian $\mathbf{Gr}(2,5)\subseteq \mathbf{P}^9$ by a linear subspace of codimension 2.  See also \cite{PZ17} for the case of Fano-Mukai fourfolds of genus $10$ as compactifications of $\mathbf{A}^4$. Three families of Fano-Mukai fourfolds are realized as complete intersections in some projective spaces and therefore can be discarded by looking at their automorphism groups.

\subsection{Fano fourfolds with $B_2\geq 2$}

Besides the toric case, we can say that additive Fano fourfolds of Picard number $\rho_X\geq 2$ cannot have large {\it Lefschetz defect}. Let us recall the following definition by Casagrande \cite{Cas12}: Let $X$ be a smooth Fano manifold, the Lefschetz defect of $X$ is defined by
$$\delta_X=\max \{\dim \ker \mbox{H}^2(X,\mathbf{R})\to\mbox{H}^2(D,\mathbf{R}),\; D \mbox{ prime divisor in }X \}, $$
where for a prime divisor $D\subseteq X$ we denote by $\mbox{H}^2(X,\mathbf{R})\to\mbox{H}^2(D,\mathbf{R})$ the restriction map. 

Casagrande proved in \cite[Theorem 1.1]{Cas12} that $0\leq \delta_X \leq 8$ for any smooth Fano variety, and if $\delta_X\geq 4$ then $X\cong S\times T$ where $S$ is a del Pezzo surface with $\rho_X \geq \delta_X+1$ and $T$ is a smooth Fano variety. In particular, if $X$ is an additive Fano fourfold we can deduce by Blanchard's lemma that the latter case is not possible (cf. Example \ref{example:del Pezzo}). In other words, if $X$ is an additive Fano fourfold then $0\leq \delta_X \leq 3$. We may refer the reader to \cite{Cas13a,Cas13b,Cas14} for bounds for the Picard number of $X$ depending on the value of $\delta_X$. 

We also note that by the results of Romano \cite[Theorem 1.1]{Rom16} if $X$ admits a conic bundle structure $f:X\to Y$, i.e. $f:X\to Y$ is a fiber type contraction whose fibers are isomorphic to (eventually singular) plane conics, then $1\leq \rho_X-\rho_Y \leq 8$, and if $\rho_X-\rho_Y \geq 4$ then $\delta_X\geq 4$ and $X\cong S\times T$ where $S$ is a del Pezzo surface with $\rho_X \geq \delta_X+1$ and $T$ is a smooth Fano variety. Again, we can deduce that if $X$ is an additive Fano fourfold admitting a conic bundle structure $f:X\to Y$ then $1\leq \rho_X-\rho_Y \leq 3$.

\appendix
\section{Polytopes incribed in a rectangle}\label{appendix}

The purpose of this Appendix is to complete the proof of Proposition \ref{prop:toric case}. Namely, we prove that the toric Fano threefolds $\III_{29}, \IV_{10}, \IV_{11}$ and $\IV_{12}$ are equivariant compactifications of $\Ga^3$.

Let us recall\footnote{See Definition \ref{defi:inscribed polytope}.} that a lattice polytope $P\subseteq M_\mathbf{R}$ is {\it inscribed in a rectagle} if there is a vertex $v_0\in P$ such that
\begin{enumerate}
 \item the primitive vectors on the edges of $P$ containing $v_0$ form a basis $e_1,\ldots,e_n$ of the lattice $M$;
 \item for every inequality $\langle p, x \rangle \leq a$ on $P$ that corresponds to a facet of $P$ not passing through $v_0$ we have $\langle p,e_i \rangle \geq 0$ for all $i=1,\ldots,n$.
\end{enumerate}

In view of Theorem \ref{theo:additive toric varieties}, we need to determine which Fano polytopes are inscribed in a rectangle in order to classify smooth toric Fano threefolds which are equivariant compactifications of $\mathbf{G}_a^3$.

\begin{remark}\label{rema: appendix} Let us note that condition (1) is automatically fulfilled in our case since the associated toric varieties are smooth. Moreover, we note that the points $p\in N_\mathbf{R}$ appearing in the inequalities $\langle p, x \rangle \leq a$ in condition (2) correspond to normal {\it exterior} vectors to the facets of $P$ and that we can always suppose $a=1$. Therefore, we only need to consider the normal exterior vectors defined by the facets of the polytopes in order to verify condition (2). It is a classical fact that the normal {\it interior} vectors are in fact the primitive generators of the $1-$dimensional cones in the fan $\Delta_P \subseteq N_{\mathbf{R}}$, i.e., we only need to check condition (2) for the set $p_1=-u_1,\ldots,p_{\rho(X)+3}=-u_{\rho(X)+3}$ where $u_1,\ldots,u_{\rho(X)+3}$ are the primitive generators\footnote{By abuse of notation, we identify the $1-$dimensional cone $\rho=\mathbf{R}_{\geq 0}u_\rho \in \Delta_P(1)\subseteq N_\mathbf{R}$ with its primitive generator $u_\rho\in N$} of the rays in $\Delta_P(1)$.
\end{remark}

In view of Remark \ref{rema: appendix} above, for each smooth toric Fano threefold $\III_{29}, \IV_{10}, \IV_{11}$ and $\IV_{12}$, we need to determine its associated polytope and its associated fan. We will follow the Graded Ring Database's section which is based on \cite{Kas06}.

\begin{enumerate}
 \item $\III_{29}$ with polytope given by the convex hull of the following points
 \begin{equation*}\begin{split}\{ & (-1,-1,-1), (-1,-1,3), (-1,3,-1), (1,-1,-1), (0,-1,2), (1,-1,0), \\ & (0,2,-1), (1,0,-1) \}\end{split}\end{equation*} 
  \begin{claim*}
  The above polytope {\bf is inscribed} in a rectangle.
 \end{claim*}
 \begin{proof}
  The primitive generators of the $1-$dimensional cones in the fan $\Delta_P$ are given by 
  $$\Delta_P(1)=\{(1,0,0), (0,1,0), (0,0,1), (-1,-1,-1), (-1,0,0), (-2,-1,-1) \}.$$
  Let $v_0=(-1,-1,-1)$. The facets not containing $v_0$ correspond to vectors $u_i=-p_i \in \Delta_P(1)$ such that $\langle p_i, v_0\rangle \neq 1$. Namely, to $p_4 = (1,1,1) $, $p_5=(1,0,0)$ and $p_6 = (2,1,1)$. The edges of $P$ containing $v_0$ are
  $$[v_0;(1,-1,-1)],[v_0;(-1,3,-1)],[v_0;(-1,-1,3)]. $$  
  The primitive vectors on these edges are given by $e_1 = (1,0,0)$, ${e_2 = (0,1,0)}$ and $e_3 = (0,0,1)$, for which we have $\langle p_4,e_i \rangle \geq 0$, $\langle p_5,e_i \rangle \geq 0$ and $\langle p_6,e_i \rangle \geq 0$.
 \end{proof}
\begin{figure}[H]
\centering
\begin{tikzpicture}[tdplot_main_coords,scale=0.5,line join=bevel]
\coordinate (A1) at (-1,-1,-1) ;
\coordinate (A2) at (-1,-1,3) ;
\coordinate (A3) at (-1,3,-1) ;
\coordinate (A4) at (1,-1,-1) ;
\coordinate (A5) at (0,-1,2) ;
\coordinate (A6) at (1,-1,0) ;
\coordinate (A7) at (0,2,-1) ;
\coordinate (A8) at (1,0,-1) ;

\fill  (A1)  circle[radius=4pt] node[left=0pt] {$v_0$};

\draw (A1) -- (A2) -- (A3) -- cycle  ;
\draw (A1) -- (A2) -- (A5) -- (A6) -- (A4) -- cycle  ;
\draw (A1) -- (A4) -- (A8) -- (A7) -- (A3) -- cycle  ;
\draw (A4) -- (A6) -- (A8) -- cycle  ;
\draw (A5) -- (A6) -- (A8) -- (A7) -- cycle  ;
\draw (A7) -- (A5) -- (A2) -- (A3) -- cycle  ;
\draw [fill opacity=0.7,fill=black!80!white] (A4) -- (A6) -- (A8) -- cycle  ;
\draw [fill opacity=0.7,fill=black!80!white] (A5) -- (A6) -- (A8) -- (A7) -- cycle  ;
\draw [fill opacity=0.7,fill=black!80!white] (A7) -- (A5) -- (A2) -- (A3) -- cycle  ;
\end{tikzpicture}
\caption{Fano polytope $\III_{29}$} \label{fig:P10}
\end{figure}
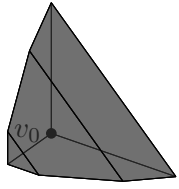

 \item $\IV_{11}$ with polytope given by the convex hull of the following points
 \begin{equation*}\begin{split}\{ & (-1,-1,-1), (-1,-1,1), (-1,2,-1), (-1,2,1), (1,-1,-1), (0,-1,1), \\ & (1,-1,0), (1,0,-1), (0,1,1), (1,0,0) \} \end{split}\end{equation*} 
  \begin{claim*}
  The above polytope {\bf is inscribed} in a rectangle.
 \end{claim*}
 \begin{proof}
  The primitive generators of the $1-$dimensional cones in the fan $\Delta_P$ are given by 
  $$\Delta_P(1)=\{(1,0,0), (0,1,0), (0,0,1), (-1,-1,0), (0,0,-1), (-1,0,0), (-1,0,-1) \}.$$
  Let $v_0=(-1,-1,-1)$. The facets not containing $v_0$ correspond to vectors $u_i=-p_i \in \Delta_P(1)$ such that $\langle p_i, v_0\rangle \neq 1$. Namely, to $p_4 = (1,1,0)$, $p_5=(0,0,1)$, $p_6=(1,0,0)$ and $p_7=(1,0,1)$. The edges of $P$ containing $v_0$ are
  $$[v_0;(1,-1,-1)],[v_0;(-1,2,-1)],[v_0;(-1,-1,1)]. $$  
  The primitive vectors on these edges are given by $e_1 = (1,0,0)$, ${e_2 = (0,1,0)}$ and $e_3 = (0,0,1)$, for which we have $\langle p_4,e_i \rangle \geq 0$, $\langle p_5,e_i \rangle \geq 0$, $\langle p_6,e_i \rangle \geq 0$ and $\langle p_7,e_i \rangle \geq 0$.
 \end{proof}
 \begin{figure}[H]
\centering
\begin{tikzpicture}[tdplot_main_coords,scale=0.5,line join=bevel]
\coordinate (A1) at (-1,-1,-1) ;
\coordinate (A2) at (-1,-1,1) ;
\coordinate (A3) at (-1,2,-1) ;
\coordinate (A4) at (-1,2,1) ;
\coordinate (A5) at (1,-1,-1) ;
\coordinate (A6) at (0,-1,1) ;
\coordinate (A7) at (1,-1,0) ;
\coordinate (A8) at (1,0,-1) ;
\coordinate (A9) at (0,1,1) ;
\coordinate (A10) at (1,0,0) ;

\fill  (A1)  circle[radius=4pt] node[left=8pt, above=-4pt] {$v_0$};

\draw (A1) -- (A3) -- (A8) -- (A5) -- cycle  ;
\draw (A5) -- (A7) -- (A6) -- (A2) -- (A1) -- cycle  ;
\draw (A2) -- (A4) -- (A3) -- (A1) -- cycle  ;
\draw (A5) -- (A7) -- (A10) -- (A8) -- cycle  ;
\draw (A9) -- (A10) -- (A7) -- (A6) -- cycle  ;
\draw (A8) -- (A10) -- (A9) -- (A4) -- (A3) -- cycle  ;
\draw (A2) -- (A4) -- (A9) -- (A6) -- cycle  ;
\draw [fill opacity=0.7,fill=black!80!white] (A5) -- (A7) -- (A10) -- (A8) -- cycle  ;
\draw [fill opacity=0.7,fill=black!80!white] (A9) -- (A10) -- (A7) -- (A6) -- cycle  ;
\draw [fill opacity=0.7,fill=black!80!white] (A8) -- (A10) -- (A9) -- (A4) -- (A3) -- cycle  ;
\draw [fill opacity=0.7,fill=black!80!white] (A2) -- (A4) -- (A9) -- (A6) -- cycle  ;
\end{tikzpicture}
\caption{Fano polytope $\IV_{11}$} \label{fig:P15}
\end{figure}
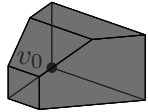
 \item $\IV_{12}$ with polytope given by the convex hull of the following points
 \begin{equation*}\begin{split}\{ & (-1,-1,-1), (-1,-1,3), (-1,1,-1), (-1,1,1), (1,-1,-1), (1,-1,1), \\ & (1,0,-1), (0,1,-1), (1,0,0), (0,1,0) \} \end{split}\end{equation*} 
  \begin{claim*}
  The above polytope {\bf is inscribed} in a rectangle. 
 \end{claim*}
 \begin{proof}
  The primitive generators of the $1-$dimensional cones in the fan $\Delta_P$ are given by 
  $$\Delta_P(1)=\{(1,0,0), (0,1,0), (0,0,1), (-1,-1,-1), (-1,0,0), (0,-1,0), (-1,-1,0) \}.$$
  Let $v_0=(-1,-1,-1)$. The facets not containing $v_0$ correspond to vectors $u_i=-p_i \in \Delta_P(1)$ such that $\langle p_i, v_0\rangle \neq 1$. Namely, to $p_4 = (1,1,1)$, $p_5=(1,0,0)$, $p_6=(0,1,0)$ and $p_7=(1,1,0)$. The edges of $P$ containing $v_0$ are
  $$[v_0;(1,-1,-1)],[v_0;(-1,1,-1)],[v_0;(-1,-1,3)]. $$  
  The primitive vectors on these edges are given by $e_1 = (1,0,0)$, ${e_2 = (0,1,0)}$ and $e_3 = (0,0,1)$, for which we have $\langle p_4,e_i \rangle \geq 0$, $\langle p_5,e_i \rangle \geq 0$, $\langle p_6,e_i \rangle \geq 0$ and $\langle p_7,e_i \rangle \geq 0$.
 \end{proof}
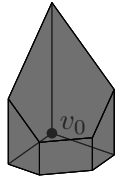
\begin{figure}[H]
\centering
\begin{tikzpicture}[tdplot_main_coords,scale=0.5,line join=bevel]
\coordinate (A1) at (-1,-1,-1) ;
\coordinate (A2) at (-1,-1,3) ;
\coordinate (A3) at (-1,1,-1) ;
\coordinate (A4) at (-1,1,1) ;
\coordinate (A5) at (1,-1,-1) ;
\coordinate (A6) at (1,-1,1) ;
\coordinate (A7) at (1,0,-1) ;
\coordinate (A8) at (0,1,-1) ;
\coordinate (A9) at (1,0,0) ;
\coordinate (A10) at (0,1,0) ;

\fill  (A1)  circle[radius=4pt] node[right=8pt, above=-4pt] {$v_0$};

\draw (A1) -- (A3) -- (A8) -- (A7) -- (A5) -- cycle  ;
\draw (A5) -- (A6) -- (A2) -- (A1) -- cycle  ;
\draw (A2) -- (A4) -- (A3) -- (A1) -- cycle  ;
\draw (A5) -- (A7) -- (A9) -- (A6) -- cycle  ;
\draw (A9) -- (A10) -- (A8) -- (A7) -- cycle  ;
\draw (A8) -- (A10) -- (A4) -- (A3) -- cycle  ;
\draw (A2) -- (A4) -- (A10) -- (A9) -- (A6) -- cycle  ;
\draw [fill opacity=0.7,fill=black!80!white] (A5) -- (A7) -- (A9) -- (A6) -- cycle  ;
\draw [fill opacity=0.7,fill=black!80!white] (A9) -- (A10) -- (A8) -- (A7) -- cycle  ;
\draw [fill opacity=0.7,fill=black!80!white] (A8) -- (A10) -- (A4) -- (A3) -- cycle  ;
\draw [fill opacity=0.7,fill=black!80!white] (A2) -- (A4) -- (A10) -- (A9) -- (A6) -- cycle  ;
\end{tikzpicture}
\caption{Fano polytope $\IV_{12}$} \label{fig:P16}
\end{figure}
\end{enumerate}

\section*{Acknowledgements}
We are very grateful to Mikhail \textsc{Zaidenberg} for many valuable suggestions and constant encouragement during the preparation of this article. We would like to thank Ivan \textsc{Arzhantsev} and Yuri \textsc{Prokhorov} for important suggestions. We also thank Michel \textsc{Brion}, Pierre-Emmanuel \textsc{Chaput}, Adrien \textsc{Dubouloz}, Andrea \textsc{Fanelli}, Bruno \textsc{Laurent} and Ronan \textsc{Terpereau} for fruitful discussions. Last but not the least we appreciate the stimulating working environment provided by the Fourier Institute where most part of this work was done. 

The second author is founded by Hua Loo-Keng Center for Mathematical Sciences, AMSS, CAS.

{\small \bibliography{Ga3}{}}

\end{document}